\documentclass{amsart}              

\usepackage{comment}
\usepackage{lipsum}
\usepackage{hyperref} 
\usepackage{amssymb}
\usepackage{tikz} 
\usepackage{amsfonts}
\usepackage{dsfont}
\usepackage{mathtools} 
\usepackage{stmaryrd}

\usepackage{graphicx}
\usepackage[caption=false]{subfig}
\usepackage{epstopdf}
\usepackage{algorithmic}
\usepackage{amsopn}
\usepackage{enumitem}
\usepackage{amsmath}
\usepackage{amsthm}
\usepackage{amsbsy}
\usepackage{psfrag}
\usepackage{tikz}
\usetikzlibrary{calc}
\usepackage{graphics}
\usepackage{graphicx}
\usepackage{xcolor}
\usepackage{soul}       

\ifpdf
  \DeclareGraphicsExtensions{.eps,.pdf,.png,.jpg}
\else
  \DeclareGraphicsExtensions{.eps}
\fi

\definecolor {uclablue}   {RGB} {39,116,174}
\definecolor {masongreen} {RGB} {0,102,51}
\newcommand{\red}[1]{{\color{red} #1}}

\newcommand{\Gammaint}[0]{\Gamma^{\rm int}_h}        
\newcommand{\triang}[0]{\mathcal{T}_h}           
\newcommand{\ehat}[0]{\hat{e}}              
\newcommand{\nhat}[0]{\hat{n}}              
\newcommand{\tauhat}[0]{\hat{\tau}}              

\newcommand{\jump}[1]{\left\llbracket #1 \right\rrbracket}  
\newcommand{\avg}[1]{\left\{\mkern-5mu\left\{ #1 \right\}\mkern-5mu\right\}} 

\newcommand{\RR}[0]{\mathbb{R}}
\newcommand{\PP}[0]{\mathbb{P}}

\newcommand{\NN}[0]{\mathbb{N}}

\newcommand{\cC}[0]{\mathcal{C}}
\newcommand{\cE}[0]{\mathcal{E}}

\newcommand{\cT}[0]{\mathcal{T}}
\newcommand{\cV}[0]{\mathcal{V}}

\def\O{\Omega}

\def\cV{\mathcal{V}}

\def\cT{\mathcal{T}}
\def\cE{\mathcal{E}}

\def\sjump#1{[\hskip -1.5pt[#1]\hskip -1.5pt]}

\usepackage[textsize=small]{todonotes}
\newtheorem{theorem}{Theorem}[section]
\newtheorem{lemma}[theorem]{Lemma}

\theoremstyle{definition}

\theoremstyle{remark}
\newtheorem{remark}[theorem]{Remark}

\numberwithin{equation}{section}

\begin{document}

\title[Error estimates for thin sheet folding]{A Posteriori and a priori error estimates for 
linearized thin sheet folding}


\author[H. Antil]{Harbir Antil}
\address{Department of Mathematics and Center for Mathematics and Artificial Intelligence, 
George Mason University, Fairfax, VA, 22030, USA}
\email{hantil@gmu.edu}
\thanks{This work is partially supported by the Office of Naval Research (ONR) under Award NO: N00014-24-1-2147,
NSF grant DMS-2408877, the Air Force Office of Scientific Research (AFOSR) under Award NO: FA9550-22-1-0248.}

\author[S. P. Carney]{Sean P.\ Carney}
\address{Department of Mathematics, Union College, Schenectady, NY, 12308, USA}
\email{carneys@union.edu}
\thanks{}

\author[R. Khandelwal]{Rohit Khandelwal}
\address{Department of Mathematics and Center for Mathematics and Artificial Intelligence, 
George Mason University, Fairfax, VA, 22030, USA}
\curraddr{}
\email{rkhandel@gmu.edu}
\thanks{}



\begin{abstract}
We describe \emph{a posteriori} error analysis for a discontinuous Galerkin
method for a fourth order elliptic interface problem 
that arises from a linearized model of thin sheet folding. 
The primary contribution is a local
efficiency bound for an estimator that measures
the extent to which the interface conditions along the fold are satisfied,
which is accomplished by constructing a novel edge bubble function. 
We subsequently conduct
a \emph{medius} analysis to obtain improved a \emph{priori} error estimates 
under the minimal regularity assumption on the exact solution. 
The performance of the method is illustrated by
numerical experiments.

\end{abstract}

\maketitle


\section{Introduction}
\label{sec:introduction}
Thin foldable structures can be found in a variety of engineering systems and
natural phenomena. These structures exhibit a balance between 
flexibility, stability and retractability, for example in 
the hind wings of ladybird beatles (i.e.\ ``ladybugs'') 
and some roach species \cite{saito2017investigation}. Examples from 
engineering science include the James Webb Space Telescope \cite{james_webb}
and Starshade technology for the detection of exoplanets \cite{starshade} from NASA; 
these structures were designed to fold up into a configuration compact 
enough to be launched from earth without sacrificing 
their scientific utility when deployed in space. 
Additional examples come from architecture \cite{schleicher2015methodology}, 
sheet metal pressing and wrapping \cite{paulsen2019wrapping}, and both 
origami and kirigami \cite{choi2021compact,demaine2007geometric,liu2021origami}. 

Mathematical models of thin foldable structures 
fundamentally
originate from the physics of three-dimensional 
hyperelastic materials. If $\delta$ denotes the width of a thin sheet
occupying a simply connected region $\Omega_{\delta} \subset \RR^3$ 
that includes some ``creased'', or folded region 
$\cC_{\epsilon}\subset \Omega_{\delta}$ where it is weakened (resulting in a 
so-called ``prepared material''), 
then the hyperelastic energy associated to some 
 deformation vector $y: \Omega_{\delta} \to \RR^3$ is
\begin{equation}\label{eq:pre_Gamma_E}
E_{\delta,\epsilon}(y) = \int_{\Omega_{\delta}} W_{\epsilon}(x,\nabla y(x)) \, dx - \int_{\Omega_{\delta}} f(x)\cdot y(x) \,dx.
\end{equation}
Here $W_{\epsilon}$ is some physically valid, material dependent free
energy density, and $f$ is some prescribed forcing term. 
Building off of the work in \cite{friesecke2002theorem}, the authors in \cite{bartels2022modeling}
show that in the 
limit as the sheet width $\delta$ and the measure of 
the crease region $|\cC_{\epsilon}|$ both vanish, 
$E_{\epsilon,\delta}$ $\Gamma$-converges \cite[Definition 12.1.1]{attouch2006variational}
to 
\begin{equation}\label{eq:full_E}
E(y) = \frac12 \int_{\Omega \setminus \cC} | D^2 y(x) |^2 \, dx - \int_{\Omega} f(x) \cdot y(x) \, dx , 
\end{equation}
where the crease $\cC$ here is a one-dimensional curve contained in $\Omega$, a two dimensional, 
simply connected open subset of $\RR^3$, $D^2$ denotes the Hessian, and $|\cdot |$ denotes
the Euclidean norm for rank-3 tensors. 
The space of admissible functions over which to minimize \eqref{eq:full_E} is 
$$
\big\{ y \in [H^2(\Omega\setminus \cC) \cap W^{1,\infty}(\Omega)]^3 \, \big\vert  \, 
y = \mu, \,\nabla y = \Psi \text{ on } \partial_D \Omega, \, (\nabla y)^\top \nabla y = I 
\text{ a.e.\ in } \Omega \big\},
$$
which enforces the isometry constraint $(\nabla y)^\top \nabla y =I$ and encodes 
clamped boundary conditions $y = \mu$ and $\nabla y = \Psi$ on a subset $\partial_D \Omega$ 
of the domain boundary $\partial \Omega$.
For compatibility, it is required that $\nabla \mu = \Psi$ on $\partial_D \Omega$. 
Notice that admissible functions need only be $H^2$ regular on $\Omega\setminus \cC$, which 
allows for jumps in the deformation gradient $\nabla y$ along the fold $\cC$.

In the past decade, several authors have considered the problem of numerically computing 
minimizers to \eqref{eq:pre_Gamma_E} and \eqref{eq:full_E} 
\textit{in the absence of} a fold.
For example, the works 
\cite{bartels2013approximation} and \cite{bartels2017bilayer} 
compute numerical approximations
using discrete
Kirchhoff triangles and quadrilaterals, respectively. Approximation schemes based
on the interior penalty discontinuous Galerkin (DG) method \cite{bonito2021dg} 
and the local discontinuous Galerkin method 
(LDG) \cite{bonito2023numerical,bonito2024gamma} have also been proposed and analyzed. 

More recently, numerical methods for the problem \textit{with} 
a fold have been devised, for example, in
  \cite{bartels2022modeling,bonito2023numerical} and \cite{bonito2024finite}. 
The former works are based on 
an LDG method, while the latter method uses continuous 
finite element basis functions and also 
considers an additional term in the energy functional which accounts for possible material stretching. 

Directly related to the content in the present article is
the work in \cite{MR4699572}, where
the authors analyzed an interior penalty DG 
method for a linearized version of the energy \eqref{eq:full_E} 
valid in the regime of small deformations. 
In this setting,
the vector $y$ is assumed to be a perturbation of the identity map
in a single direction, typically taken to be the vertical direction. 
Under this assumption, the isometry constraint of the full nonlinear model
can be omitted, and it suffices to consider the vertical component 
$u := y_3 : \Omega \to \RR$ of the deformation \cite[Chapter 8]{MR3309171}. 
The linearized folding model energy becomes
\begin{equation} \label{eq:linear_E}
E_{\rm linear}(u) = \frac12 \int_{\Omega\setminus \cC} |D^2 u(x) |^2 \, dx - \int_{\Omega} f(x) \cdot e_3 \, u(x) \, dx,
\end{equation}
where $|\cdot|$ now denotes the Euclidean norm on rank-2 tensors, and
the space of admissible minimizers is 
$$
\big\{ u \in H^2(\Omega\setminus \cC) \cap H^{1}(\Omega) \, \big\vert  \, 
u = g \text{ and } \nabla u = \Phi \text{ on } \partial_D \Omega \big\}; 
$$
as before, $\nabla g = \Phi$ is required on $\partial_D \Omega$ for compability. 

Despite this recent progress in the numerical analysis of 
folding models, only a few authors have begun to address the 
issue of how to design 
\emph{adaptive} numerical methods. 
In one example, the authors in 
\cite{caboussat2022anisotropic}
proposed an adaptive mesh refinement routine for the related problem of
 approximating so-called orthogonal maps, which arise in mathematical models of 
origami. While the numerical results are promising, no rigorous analysis 
was presented.

The only other work in this direction appears to be 
a preliminary \emph{a posteriori}
error analysis of an interior penalty DG method 
discretization of the linearized folding model 
resulting from \eqref{eq:linear_E}
that was presented
in the doctoral thesis \cite{tschernernumerical}.
However, the thesis only provides a reliability bound, i.e.\
suitably defined estimators were shown to form an upper bound 
for the discrete error in an energy norm. The reverse inequality--an
efficiency bound--was not shown. 

As a first step towards robust \textit{a posteriori} error estimates of more sophisticated
folding models, 
we provide in the present article 
the missing local efficiency 
estimates for the interior penalty DG discretization of the 
linear folding model considered in 
\cite{MR4699572} and \cite{tschernernumerical}.

With reliability and efficiency estimates in hand, we also 
obtain improved \textit{a priori} error estimates 
via a \textit{medius} analysis \cite{MR2684360}. 
The \emph{a priori} estimates previously obtained in \cite{MR4699572} rely on Galerkin orthogonality, 
which requires the solution to the continuous (i.e.\ infinite dimensional) problem 
to be $H^4(\Omega\setminus\cC) \cap H^1(\Omega)$ regular. Such high regularity may not always be realistic \cite[Appendix A]{MR3022211}, 
as we observe in numerical examples here. Instead, our improved
\textit{a priori} estimates only require $H^2(\Omega\setminus\cC)\cap H^1(\Omega)$ regularity.

The Euler-Lagrange equation for the linearized energy \eqref{eq:linear_E} 
considered here is (in its strong form) a fourth order, biharmonic equation with 
a set of interface conditions along the crease $\cC$. 
Although adaptive, nonconforming finite element methods have been developed and analyzed for
the classical biharmonic problem, for example in
\cite{MR2755946,brenner2010posteriori,
fraunholz2015convergence,dominicus2024convergence}, 
such analysis of fourth order interface problems  
is still open. 
This is analogous to the second order case; 
while the literature on adaptive finite element methods for second
order elliptic problems is vast, there are comparatively few studies  
on adaptive methods for second order elliptic interface problems. 
In the context of interior penalty DG methods, rigorous analysis was 
carried out in \cite{cai2011discontinuous} and \cite{cangiani2018adaptive}, 
while computational aspects were addressed in 
\cite{annavarapu2012robust,saye2019efficient}.

Most of the error estimators here have been previously introduced 
for \textit{a posteriori} 
analyses for the classical biharmonic problem. 
The interface conditions along the fold $\cC$ result
in a new error estimator, first introduced in \cite{tschernernumerical}, 
for which we prove a new local efficiency bound. 
The estimator involves the average of a second derivative of 
the deformation along the fold; 
we bound this estimator by the PDE residual in neighboring elements 
by constructing a novel edge bubble function, which may be useful for
\textit{a posteriori} analysis of other interface problems.

In line with
previous analysis of DG approximations for interface problems
\cite{cangiani2013discontinuous,cangiani2018adaptive,bonito2024finite,MR4699572,tschernernumerical}, 
we consider here a fitted discretization, 
i.e.\ we assume that numerical mesh can exactly represent the crease geometry: 
$\cC = \cC_h$. 
As in \cite{bartels2022modeling,bonito2023numerical}, our analysis further 
assumes the case of a piecewise linear folding curve $\cC$. 
While there are many interesting deformations that can result from piecewise linear 
folds, this is, of course, a relatively strong assumption. For technical reasons, however, it is 
difficult to relax this assumption for \textit{a posteriori} analysis, as
it is not clear how to construct recovery, or ``enriching'' operators that map from the DG 
space to a conforming space with higher regularity, 
in the presence of interface. This difficulty was previously 
pointed out in \cite[Section 11]{cangiani2018adaptive} in the context of a second order interface problem, 
which is still an open question.
The issue is more acute for the fourth order problem considered here, as
enriching operators for second order problems map to $C^0$ functions, while 
for fourth order problems they map to $C^1$ functions \cite[Section 1]{MR543934}. For 
our analysis in particular, the target $C^1$ space must include normal derivatives as 
degrees of freedom (DoF), yet there is currently no well-defined framework for isoparametric finite elements that preserve such derivative-based DoFs \cite{MR1300112}.

The fitted discretization assumption leads to a geometric consistency
error (a kind of ``variational crime'' \cite{strang1972variational}) 
that was quantified in \cite[Theorem 4.3]{MR4699572}; 
if the interface is approximated
as a piecewise polynomial map of degree $m$, then the geometric error will decay like 
$h^{m-1}$. As remarked in \cite{MR4699572}, 
 this error scaling is consistent with the implications of Babu\v{s}ka's paradox; 
see also \cite{bartels2024necessary,bartels2025babu} for more details. 
Consistent with the results in \cite{MR4699572,tschernernumerical}, however,
we do not detect this error in any of our numerical experiments here. 

The rest of the manuscript is organized as follows. 
Sections \ref{sec:model_problem} and \ref{sec:dg_method} 
detail the model PDE problem and DG discretization, respectively. 
We introduce the novel bubble function in Section \ref{sec:bubble} and 
prove the key properties therein. 
Section \ref{sec:reliability} recapitulates the reliability estimates from 
\cite{tschernernumerical}, and in Section \ref{sec:efficiency}, we employ the 
new bubble function to prove local efficiency estimates. 
Section \ref{sec:medius} contains the improved \textit{a priori} 
estimates obtained with a \textit{medius} analysis. Finally, we present several 
numerical examples in Section \ref{sec:numerical_examples} which 
substantiate the performance of the \emph{a posteriori} error estimators
over adaptive mesh refinement, and we offer concluding remarks in
Section \ref{sec:conclusions}.


\section{Model problem}
\label{sec:model_problem}
We now summarize the linear folding model problem introduced in 
\cite{MR4699572} that we consider throughout the rest of the manuscript. 
Let $\Omega \subset \RR^2$ be an open, bounded, polygonal Lipschitz domain 
that is partitioned into two distinct subdomains $\Omega_1$ and $\Omega_2$
by a Lipschitz curve $\cC$, so that $\Omega = \Omega_1\cup \cC \cup \Omega_2$, 
as depicted in Figure \ref{fig:cartoon_domain}. For $\omega \subseteq \Omega$, we recall the Hilbert space 
$H^r(\omega)$ $(r \in \mathbb{N})$, which is the set of all $L^2(\omega)$ functions 
whose distributional derivatives up to order $r$ are in $L^2(\omega).$ Denote 
by $|\cdot|_{r,\omega}$, the semi-norm on the space $H^r(\omega)$ and by $H^r_0(\omega)$, 
the set of all functions in $H^r(\omega)$ whose traces vanish up to order $r-1$. 
We refer \cite{MR2373954,MR0520174} for more details.

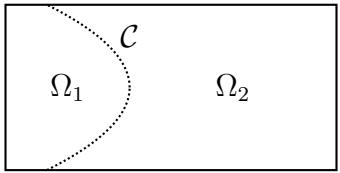
\begin{figure}\label{fig:example_domain}
    \begin{tikzpicture}[xscale=0.55, yscale=0.55]
    
    \draw[thick] (0,0) rectangle (8,4);
    


    \draw[thick, smooth, densely dotted, domain=0:2, variable=\y] plot ({-0.5*\y^2+3}, {\y+2}); 
    \draw[thick, smooth, densely dotted, domain=-2:0, variable=\y] plot ({0.5*\y^2+3}, {\y+2}); 

    \node at (1.5,2) {\Large{$\Omega_1$}};
    \node at (5.5,2) {\Large{$\Omega_2$}};
    \node at (3.,3.25) {\Large{$\cC$}};
    
    \end{tikzpicture}
    \caption{Example domain $\Omega = \Omega_1 \cup \cC \cup \Omega_2$.}
    \label{fig:cartoon_domain}
 \end{figure}

Given $f \in L^2(\Omega)$\footnote{Note the abuse of notation, 
as here $f$ is scalar-valued, while in \eqref{eq:pre_Gamma_E}, \eqref{eq:full_E} and \eqref{eq:linear_E}
$f$ is vector-valued.}, consider the problem of minimizing 
\begin{equation}\label{eq:model_energy}
\cE(u) := \frac12 \int_{\Omega \setminus \cC} |D^2 u(x)|^2 \, dx  - \int_{\Omega} f(x)u(x)\, dx
\end{equation}
over the set of admissible functions
\begin{equation}\label{eq:admissible_fns}
\cV(g, \Phi) := \Big\{ v \in H^2(\Omega_1 \cup \Omega_2) \cap H^1(\Omega) 
\Big\vert \, v = g, \nabla v = \Phi \text{ on } \partial_D \Omega \Big\}, 
\end{equation}
where $\partial_D \Omega$ is a nonempty subset of $\partial \Omega$. 
We assume $g \in H^{3/2}(\partial_D \Omega)$  and
$\Phi \in [H^{1/2}(\partial_D \Omega)]^2$ are the traces
of functions $\tilde g\in H^2(\Omega)$ and $\tilde \Phi \in [H^1(\Omega)]^2$, and that $\nabla g = \Phi$
on $\partial_D\Omega$. Notice that 
members of $\cV(g,\Phi)$ are allowed to fold along $\cC$, 
as they are globally in $H^1$ but only $H^2$ on each subdomain. 

The Euler-Lagrange equation for a minimizer to \eqref{eq:model_energy} is 
given by 
\begin{equation}\label{eq:variational_form}
a(u,v)= l(v) \qquad  \forall v \in \cV(0,0), 
\end{equation}
where
\begin{equation*}\label{eq:variational_form2}
a(u,v):= \int_{\Omega\setminus\cC} D^2 u(x) : D^2 v(x) \, dx \quad \text{and} \quad  l(v) := \int_{\Omega} f(x) v(x) \, dx .  
\end{equation*}
Existence and uniqueness of a solution $u \in \cV(g,\Phi)$ for \eqref{eq:variational_form} follow from a straightforward application of the Lax-Milgram theorem. 
Using the natural boundary conditions 
\begin{equation}\label{eq:natty_bcs}
\partial_n \nabla u = D^2 u n = 0 \quad \text{and} \quad \partial_n \Delta  u = \text{div}( D^2 u n) = 0 
\end{equation}
on $\partial_N \Omega = \partial \Omega \setminus \partial_D \Omega$ (where $n$ is the outward
unit normal vector to $\partial \Omega$), the strong form of \eqref{eq:variational_form} reads
\begin{equation}\label{eq:strong_form}
\Delta^2 u = f  \qquad \text{in } \Omega. 
\end{equation}
Additionally, we have the interface conditions 
\begin{equation}\label{eq:interface_conditions}
\jump{u}=0, \quad\quad
\partial_n \nabla u\big \vert_{\Omega_i}=0, \quad\quad
\jump{\partial_n \Delta u} = 0 
\quad  \text{ on } \cC,
\end{equation}
where $\jump{u}:= u\vert_{\Omega_2} - u\vert_{\Omega_1}$ and $n$ denotes the unit normal to $\cC$ pointing from $\Omega_1$ to $\Omega_2$.
The first and third condition in \eqref{eq:interface_conditions} 
are expected for smooth solutions to the 
fourth order problem \eqref{eq:strong_form}. 
However, since elements of $\cV(g,\Phi)$ are only globally $H^1$ regular,
$\nabla u$ need not be continuous along the fold; hence, the second
condition only implies that 
the normal component of the curvature of the deformation $u$ should 
vanish at the fold.

\section{Discontinuous Galerkin method}
\label{sec:dg_method}
\subsection{Preliminaries}
We first introduce some notation used throughout the text before
defining the DG method used to approximate the solution to the 
variational problem \eqref{eq:variational_form}. We then recall some preliminary 
results used in the subsequent analysis.  

For a  given mesh parameter $h>0$,  
let $\cT_h$ be the partition of $\O$ into regular triangles \cite[Definition 4.4.13]{MR2373954} such that 
\begin{align*}
\bar{\O}=\underset{T \in \cT_h}{ \bigcup}T.
\end{align*}
 Further,  the notation $X \lesssim Y$ is used to represent 
$X \leq CY$ where $C$ is a positive generic constant independent of the mesh parameter $h$. 

The DG finite element space used to approximate the continuous 
space \eqref{eq:admissible_fns} is defined as 
\begin{align} \label{eq:dg_fe_space}
\mathcal{V}^k_h:=\big\{v_h \in L^2(\Omega)~\big| \,\, v_h|_T \in \mathbb{P}_k(T) \quad \forall T \in \cT_h \big \},
\end{align} 
where for any $T \in \cT_h$, $\mathbb{P}_k(T)$ refers to the space of 
polynomials of degree at most $k \in \NN \cup \{0\}$. 
Associated to any $T \in \cT_h$,  
let $h_T:= \text{diam } T$. Let $\Gamma_h$ denote 
the set of all edges of $\cT_h$, and for any edge 
$e \in \Gamma_h$, let $h_e$ denote the length of edge $e$.  We define
the set of interior and Dirichlet boundary edges as
\begin{align*}
\Gammaint :=\{ e \in \Gamma_h: e \subset \Omega \} \qquad \text{and} \qquad
\Gamma_h^{D}:=\{ e \in \Gamma_h: e \subset \partial_D\Omega \},
\end{align*}
respectively, and let $\tilde{\Gamma}_h:= \Gammaint \cup \Gamma_h^D$. 
As discussed in the introduction, we assume a fitted discretization, so that 
the mesh can exactly represent the fold geometry: $\cC = \cC_h$. 
\subsection{Discrete problem} 
 Following standard procedures \cite{di2011mathematical}, we now 
introduce the DG method from \cite{MR4699572}. 
%
Consider for $T_{+}, T_{-} \in \cT_h$ the interior 
edge $\Gammaint \owns e = \partial T_{+} \cap \partial T_{-}$, and
define the jump $\jump{\cdot}$ and the average $\avg{\cdot}$ of some function $v$
across $e$ as
$$
\jump{\varphi}:= \varphi\vert_{T_+} - \varphi \vert_{T_{-}},
\qquad
\avg{\varphi}:= \frac12 \left( \varphi\vert_{T_+} + \varphi \vert_{T_{-}}\right) . 
$$
For boundary edge $e \in \Gamma_h^D$ belonging to $T\in\cT_h$, we define 
$$
\jump{\varphi}:= - \varphi\vert_{T}, \qquad \avg{\varphi} := \varphi\vert_{T}. 
$$
For vector valued functions, these definitions are understood to hold 
componentwise. 

The interior penalty DG method to approximate \eqref{eq:variational_form} is then defined as: 
find $u_h \in \cV_h^k$ such that
\begin{equation}\label{eq:discrete_problem}
a_h(u_h, v_h) = l_h(v_h) \quad \forall~ v_h \in \cV_h^k,
\end{equation}
where the symmetric bilinear form $a_h : \cV_h^k \times \cV_h^k \longrightarrow \mathbb{R}$ and the 
linear form $l_h : \cV_h^k \longrightarrow \mathbb{R}$ are given by \cite{MR4699572}
\begin{align}
a_h(u_h,v_h) &:= \sum_{T \in \cT_h} \int_T D^2 u_h: D^2 v_h \, dx \nonumber \\
& \quad + \sum_{ e \in \tilde{\Gamma}_h \setminus \cC} \int_e  \avg{\partial_n  \nabla u_h} \cdot \jump{\nabla v_h } \, ds
+ \sum_{ e \in \tilde{\Gamma}_h \setminus \cC} \int_e \avg{\partial_n  \nabla v_h} \cdot \jump{\nabla u_h } \, ds \nonumber \\
& \quad - \sum_{e \in \tilde{\Gamma}_h} \int_e \avg{\partial_n \Delta u_h}  \jump{v_h} \, ds
- \sum_{e \in \tilde{\Gamma}_h} \int_e \avg{\partial_n \Delta v_h}   \jump{u_h} \, ds \nonumber \\
& \quad + \sum_{ e \in \tilde{\Gamma}_h \setminus \cC}  \frac{\gamma_1}{h_e} \int_e \jump{\nabla u_h } \cdot  \jump{\nabla v_h} \, ds
+\sum_{e \in \tilde{\Gamma}_h} \frac{\gamma_0}{h_e^3} \int_e \jump{u_h}  \jump{v_h} \, ds  \label{eq:discrete_bilinear} 
\end{align}
and 
\begin{align}
l_h(v_h) := \sum_{T \in \cT_h}  &\int_T f \, v_h \, dx
- \sum_{e \in \Gamma_h^D } \int_e \avg{\partial_n \nabla v_h} \cdot \Phi \, ds
+ \sum_{e \in \Gamma_h^D}  \int_e \avg{\partial_n \Delta v_h} g    \, ds  \nonumber \\
& \quad - \sum_{ e \in \Gamma_h^D}  \frac{\gamma_1}{h_e} \int_e \jump{\nabla v_h} \cdot \Phi   \, ds
-\sum_{e \in \Gamma_h^D} \frac{\gamma_0}{h_e^3}  \int_e \jump{v_h} \, g \, ds.  \label{eq:discrete_linear}
\end{align}
In \eqref{eq:discrete_bilinear} and \eqref{eq:discrete_linear}, 
$\gamma_0, \gamma_1 \in \RR_+$ are penalty parameters.
Notice that the discrete bilinear form does \emph{not} penalize jumps in $\nabla u_h$ across the fold $\cC$. 

Next, define  
\begin{equation}\label{eq:dg_norm1}
\| v_h \|_{DG}^2 := \sum_{T \in \cT_h} |v_h|^2_{2,T} 
+ \sum_{e \in \tilde{\Gamma}_h} \frac{\gamma_0}{h_e^3} \| \jump{v_h} \|_{L^2(e)}^2 
+ \sum_{e \in \tilde{\Gamma}_h \setminus \cC} \frac{\gamma_1}{h_e} \| \jump{\nabla v_h} \|_{L^2(e)}^2,
\end{equation}
which is a norm on the space $\cV^k_h$ for any positive numbers $\gamma_0$ and $\gamma_1$. As shown in 
\cite[Proposition 2.5]{MR4699572}, the bilinear form \eqref{eq:discrete_bilinear} is bounded 
and coercive with respect to the DG norm $\| \cdot \|_{DG}$ for sufficiently 
large $\gamma_0$ and $\gamma_1$, and the discrete problem 
\eqref{eq:discrete_problem} admits a unique solution. 

In the following subsections, we introduce the lifting and enriching operators, as well as their associated 
properties, that will be used throughout the \emph{a priori} and \emph{a posteriori} error analysis.

\subsection{The Lifting Operator}
Lifting operators were first introduced for elliptic problems in \cite{MR1765651} (see
also \cite{MR1910752} for their use in $hp$-analysis) and allow one to relate interelement discontinuities
of a finite element function $v$ and its derivatives to its values on elements $T \in\triang$.
We here introduce a {\it lifting operator } 
$L: \cV_h^k \longrightarrow [\mathcal{V}^k_h]^{2\times 2}$ that satisfies
\begin{align} \label{eq:lifting}
\sum_{T \in \mathcal{T}_h} \int_T L(v) : w  \, dx & := -  \sum_{e \in \tilde{\Gamma}_h} \int_e\sjump{v} \avg{\text{div } w \cdot n} dx 
+  \sum_{e \in \tilde{\Gamma}_h \setminus \cC } \int_e \sjump{\nabla v} \cdot \avg{ w n} ds,
\end{align}
for all $w \in [\mathcal{V}^k_h]^{2\times 2}$.
Notice that gradient jumps are neglected along the fold $\mathcal{C}$, which is consistent
with the admissible function set \eqref{eq:admissible_fns}.
\begin{remark}
An important property of $L(\cdot)$ is that it satisfies the following stability bounds:
\begin{align} \label{stabilityL1}
\|L(v)\|^2_{L^2(\Omega)} \lesssim  \sum_{e \in \tilde{\Gamma}_h} \big \|\sqrt{\frac{\gamma_0}{h_e^3}} \sjump{v}\big \|^2_{L^2(e)} 
+ \sum_{e \in \tilde{\Gamma}_h \setminus \cC} \big \|\sqrt{\frac{\gamma_1}{h_e}} \sjump{ \nabla v} \big \|^2_{L^2(e)}, 
\quad \forall v \in \mathcal{V}_h^k.
\end{align}
A proof of \eqref{stabilityL1} is analogous to that in \cite[Lemma 5.1]{MR2520159} with the modification that gradient jumps are neglected on the interface $\mathcal{C}$ in the definition of $L(\cdot)$.
\end{remark}
\begin{remark}
By the definition \eqref{eq:lifting}, we can rewrite $a_h(\cdot, \cdot)$ as follows: 
\begin{align}
a_h(v,w) 
& = \sum_{T \in \mathcal{T }_h} \int_T (D^2v :D^2w + L(v) : D^2w +  L(w) : D^2v) dx 
+  \sum_{e \in \tilde{\Gamma}_h} \frac{\gamma_0}{h_e^3} \int_e  \sjump{v} \sjump{w} ds \notag \\ 
& \qquad +  \sum_{e \in \tilde{\Gamma}_h \setminus \mathcal{C}} \frac{\gamma_1}{h_e}\int_e  \sjump{\nabla v} \cdot \sjump{\nabla w} ds \quad \forall v, w \in \mathcal{V}^k_h. \label{eq:discrete}
\end{align}
\end{remark}
\begin{remark}
For the analysis below, we will extend the domain of the lifting operator $L(\cdot)$ and both
$a_h(\cdot, \cdot)$ and $l_h(\cdot)$ to
$\mathcal{W}^k:=\mathcal{V}^k_h + \cV(g, \Phi):= 
\{a_h + b~|~a_h \in \mathcal{V}^k_h \text{ and } b \in \cV(g, \Phi)\}$.
\end{remark}

\subsection{The Enriching Operator} \label{subsec:enriching}
We now construct an enriching operator $E_h$ that maps functions in $\mathcal{V}_h^k$ onto a $C^1$-conforming space 
$$
\widetilde{\cV}_h^{k+2} := \big\{ v \in C^1(\overline{\Omega}) \, \big\vert \, \,\, v \vert_{T} \in \widetilde{\PP}_m \, \, \forall T \in \cT_h \big\}
$$
where for any $T \in \cT_h$ and $m = k+2$, 
$$
\widetilde{\PP}_{k+2} := \big \{v \in C^1(T)~\big\vert~v|_{T_i} \in \mathbb{P}_{k+2}(T_i) \quad \forall i = 1,2,3 \big \}.
$$
Here $T_1,T_2,T_3$ are three subelements of $T \in \cT_h$, as depicted in Figure \ref{fig:elements} for $k=2$. 

The degrees of freedom of $\widetilde{\cV}_h^{k+2}$ consist of the nodal function values of each $T_i$ ($1\le i \le 3$), 
all of the partial derivatives on the nodes $T$, all of the partial derivatives on the intersection point $z = T_1 \cap T_2 \cap T_3$, 
and the normal derivatives on two distinct inner points of all edges $e_1,e_2,e_3$ of $T$. 
See Figure \ref{fig:elements} for visualization. We here adapt the same definition of enriching operator as introduced in \cite[Lemma 3.7]{tschernernumerical} to our linear folding model and then, the following approximation properties hold:
\begin{lemma}
	Let $\beta \in \{0,1,2\}$ and $|\cdot|_{0,T}:= \|\cdot\|_{L^2(T)}$, then there exists 
$E_h : \mathcal{V}_h^k \longrightarrow  \widetilde{V}_h^{k+2} \cap \cV(g, \Phi)$ satisfying
	\begin{align} 
	\sum_{T   \in \cT_h} |v_h - &E_h(v_h)|^2_{\beta,T}   \nonumber \\
&\lesssim \sum_{e \in \tilde{\Gamma}_h} \|h^{\frac{1}{2}-\beta} \sjump{v_h}\|^2_{L^2(e)} 
+  \sum_{e \in \tilde{\Gamma}_h \setminus \cC} \|h^{\frac{3}{2}-\beta} \sjump{\nabla v_h}\|^2_{L^2(e)}   \label{estimate:enriching}
	\end{align}
for any $v_h \in \cV_h^k$. 
\end{lemma}
\begin{proof}
	See \cite[Lemma 3.7]{tschernernumerical} for the proof.
\end{proof}
\begin{remark}(Stability of $E_h$) Using the definitions of $E_h$ and $\| \cdot \|_{DG}$, the triangle inequality, 
and estimate \eqref{estimate:enriching} for $\beta=2$, the following immediately holds for all $ v_h \in \cV_h^k$:
	\begin{align} 
	\| E_h v_h \|^2_{DG}  & =  \sum_{T \in \cT_h} |E_h   v_h|^2_{2,T} \nonumber \\ & \leq \sum_{T \in \cT_h} |v_h|^2_{2,T} + \sum_{T \in \cT_h} |E_h v_h  -  v_h|^2_{2,T}   \lesssim \| v_h \|^2_{DG}. \label{estimate:stabilityenriching}
	\end{align} 
\end{remark}
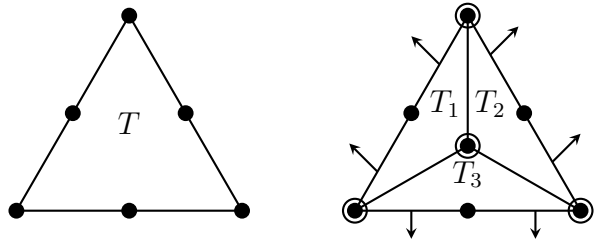
\begin{figure} 
	\begin{tikzpicture}[xscale=0.75, yscale=0.75]
	
	
	\coordinate (A) at (-6,0);
	\coordinate (B) at (-2,0);
	\coordinate (C) at (-4,3.46);
	
        \draw (-4, 1.2)node[above]{\Large{$T$}};

	\draw[thick] (A) -- (B) -- (C) -- cycle;
	
	\coordinate (AB) at ($(A)!0.5!(B)$);
	\coordinate (BC) at ($(B)!0.5!(C)$);
	\coordinate (CA) at ($(C)!0.5!(A)$);
	
	\fill (A) circle (4pt);
	\fill (B) circle (4pt);
	\fill (C) circle (4pt);
	
	\fill (AB) circle (4pt);
	\fill (BC) circle (4pt);
	\fill (CA) circle (4pt);
	
	
 \coordinate (A) at (0,0);
\coordinate (B) at (4,0);
\coordinate (C) at (2,3.46);

\draw (1.6, 1.5)node[above]{\Large{$T_1$}};
\draw (2.4, 1.5)node[above]{\Large{$T_2$}};
\draw (2.0, 0.25)node[above]{\Large{$T_3$}};

\draw[thick] (A) -- (B) -- (C) -- cycle;

\coordinate (AB) at ($(A)!0.5!(B)$);
\coordinate (BC) at ($(B)!0.5!(C)$);
\coordinate (CA) at ($(C)!0.5!(A)$);

\coordinate (AB1) at ($(A)!0.25!(B)$);
\coordinate (AB2) at ($(A)!0.8!(B)$);
\coordinate (BC1) at ($(B)!0.25!(C)$);
\coordinate (BC2) at ($(B)!0.8!(C)$);
\coordinate (CA1) at ($(C)!0.25!(A)$);
\coordinate (CA2) at ($(C)!0.8!(A)$);

\foreach \point in {A, B, C} {
	\draw[thick] (\point) circle (6pt);
}

\draw[thick,->,>=stealth,scale=0.5] (AB1) -- ++(0,-1);
\draw[thick,->,>=stealth,scale=0.5] (AB2) -- ++(0,-1);
\draw[thick,->,>=stealth,scale=0.5] (BC1) -- ++(1,1);
\draw[thick,->,>=stealth,scale=0.5] (BC2) -- ++(1,1);
\draw[thick,->,>=stealth,scale=0.5] (CA1) -- ++(-1,1);
\draw[thick,->,>=stealth,scale=0.5] (CA2) -- ++(-1,1);

\fill (A) circle (4pt);
\fill (B) circle (4pt);
\fill (C) circle (4pt);
\fill (AB) circle (4pt);
\fill (BC) circle (4pt);
\fill (CA) circle (4pt);

\coordinate (P) at (2,1.15);
\fill (P) circle (4pt);
\foreach \point in {P} {
	\draw[thick] (\point) circle (6pt);
}

\draw[thick] (P) -- (A);
\draw[thick] (P) -- (B);
\draw[thick] (P) -- (C);


	
	\end{tikzpicture}
	\caption{
$\mathbb{P}_2$ Lagrange finite element on an element $T$ (left) and corresponding 
$C^1$-conforming $\widetilde{\mathbb{P}_4}$ macro element with nodal function values (dots), partial derivatives (circles) and normal derivatives (arrows)
on $T$ and its three subelements $T_1$, $T_2$ and $T_3$. }
	\label{fig:elements}
\end{figure}


\section{Novel bubble function and its properties}
\label{sec:bubble}
In this section, we introduce the novel edge bubble function which will be subsequently used in 
proving efficiency estimates. For any interior edge $\ehat  \in \partial T_- \cap \partial T_+$,
$T \in \{T_-, T_+\}$, and $\cT_{\ehat} := T_+ \cup T_-$, let $n_+$ be the unit normal vector
pointing outwards from $T_{+}$ into $T_-$, and let $\nhat := n_+$ and $n_- := -n_+$. 
Additionally, let $\tau_+$ be the unit tangential vector to $\ehat$, and denote $\tauhat := \tau_+$ and 
$\tau_- := -\tau_+$
(cf.~Figure \ref{fig:rhombus}).

For $\alpha \in \{\tauhat, \nhat\}$, 
let $\phi_{T,\alpha} : \cT_{\ehat } \longrightarrow \mathbb{R}$ 
be a function defined by $\phi_{T,\alpha} : = \psi_{\alpha} \Lambda^4_{T,1}\Lambda^4_{T,3}$, where $\Lambda_{T,i}$ is the $i$th 
barycentric coordinate corresponding to $T$ (see \cite[Section 3.2]{MR4793681}). By definition, we have $\Lambda_{T,i}(x_j) = \delta_{ij},$ where 
$x_j$ are the nodes of $T$ (see Figure \ref{fig:rhombus}) and $\delta_{ij}$ is the Kronecker delta function. 
Without loss of generality, let $\Lambda_{T_+,1}$ and $\Lambda_{T_+,3}$ 
	vanish on $AP_+$ and $P_+B$, respectively, and let $\Lambda_{T_-,1}$ and $\Lambda_{T_-,3}$  vanish on $AP_-$ and $P_-B$, respectively. 
	We define $\psi_{\alpha}: \cT_{\ehat } \longrightarrow \mathbb{R}$ to be a continuous, piecewise affine function that assumes the value zero 
	along the common edge $\ehat $, such that $\nabla \psi_{\alpha}|_{T_{\pm}} = h_{\ehat }^{-1} \alpha_{\pm}/2$, where $\alpha_{\pm}$ is the 
	unit normal/tangential to $T_{\pm}$. 
Notice that $\nabla \psi_{\alpha}$ has a discontinuity at the common edge $\ehat $, and that, by construction, 
\begin{equation} \label{eq:avg_grad_psi}
\avg{\nabla \psi_{\nhat}}\vert_{\ehat} = \frac12 h_{\ehat}^{-1}( n_+ + n_-) = 0  
\quad \text{and} \quad
\avg{\nabla \psi_{\tauhat}}\vert_{\ehat} = \frac12 h_{\ehat}^{-1}( \tau_+ + \tau_-) = 0 . 
\end{equation} 
		\begin{figure}[!hh]
		\begin{center}
			\setlength{\unitlength}{0.3cm}
			\begin{picture}(12,6)
			\put(2,1){\line(1,0){8}} 
			\put(2,1){\line(-1,1){4}} 
			\put(6,5){\line(-1,0){8}} 
			\put(10,1){\line(-1,1){4}} 
			
			\put(2,1){\line(1,1){4}}
			
			
			\put(1,0.39){$A$} 
			\put(5.95,5.25){$B$} 
			\put(10,0.29){$P_{-}$}
			\put(-3.5,5.2){$P_{+}$} 
			\put(-0.5,3.95){$T_{+}$}
			\put(7.5,1.5){$T_{-}$} 
			
			\thicklines
			\put(4,3){\vector(1,-1){1.2}} 
			\thinlines
			\put(3.8,1.4){$\hat{n}$} 

			\thicklines
			\put(4,3){\vector(1,1){1.2}} 
			\thinlines
			\put(5.3,3.4){$\hat{\tau}$} 
			
			\put(2.90,2.9){$\ehat $}

			
			\end{picture}
			\caption{The rhombus $\cT_{\ehat}  : = T_{-}\cup T_{+}$ with its common edge $\ehat =\partial T_-\cap\partial T_+$ that 
				has initial node $A$,
				end node $B$, unit normal vector $\hat{n}$ and unit tangential vector $\tauhat$. 
				}
			\label{fig:rhombus}
		\end{center}
	\end{figure}
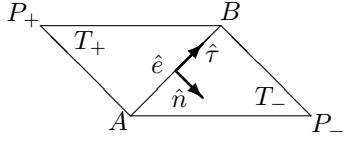
Using the above definitions of $\phi_{T,\alpha}$ and $\psi_{\alpha}$,  
we introduce the novel bubble function $\phi_{\ehat,\alpha } : \Omega \longrightarrow \mathbb{R}$ defined by
\begin{align} \label{novelbubble}
\phi_{\ehat,\alpha }|_{T} := \phi_{T,\alpha} \quad \text{where } T \in \cT_{\ehat }, \quad \text{and} \quad \phi_{\ehat,\alpha }:= 0  \quad \text{on } \Omega \setminus \cT_{\ehat }.
\end{align}
Note that, by construction, $\phi_{\ehat,\alpha}$ is an element of $H^1(\Omega) \cap H^2(\Omega_1 \cup \Omega_2)$. 
The following lemma contains some additional, useful 
properties of $\phi_{\ehat,\alpha }$.
\begin{lemma} \label{bubble:properties}
	Let $\alpha \in \{\nhat, \tauhat\}$, and let $\phi_{\ehat,\alpha}$ be defined as in \eqref{novelbubble}, then the following holds:
	\begin{align} 
	\sjump{\phi_{\ehat,\alpha }}  =\avg{ \phi_{\ehat,\alpha } } = 0 & \quad \text{ and }  \avg{ \nabla \phi_{\ehat,\alpha } }  = 0 \text{ on } \Gamma_h,  \label{bubbleproperties1} \\
	\sjump{\nabla \phi_{\ehat,\alpha }}  = 0 \text{ on } \Gamma_h \setminus \ehat  \quad  
	& \text{and } \sjump{\nabla \phi_{\ehat,\alpha }}|_{\ehat } = (h_{\ehat }^{-1} \Lambda^4_{T,1}  \Lambda^4_{T,3} \alpha)|_{{\ehat }}. \label{bubbleproperties2}
	\end{align}
\end{lemma}
\begin{proof}
	Since $\phi_{\ehat,\alpha}$ is only supported on $\cT_{\ehat}$, it suffices 
to verify \eqref{bubbleproperties1} and \eqref{bubbleproperties2} on any edge $e \in \partial T$, $T \in \cT_{\ehat}$.
For fixed $T \in \cT_h$, let $G_T: \tilde{T} \longrightarrow T$ be the standard unique, invertible 
	affine map \cite[Section 2.3]{MR0520174}, where $\tilde{T} \in \{\tilde{T}_+, \tilde{T}_- \}$ is 
	the reference element (cf.~Figure \ref{fig:reference-triangles}). We construct a  
	bubble function $\phi_{\tilde{T},\tilde\alpha}$ on $\tilde{T}$,  
	 prove \eqref{bubbleproperties1} and \eqref{bubbleproperties2} on $\tilde{T}$, and then set $\phi_{T,\alpha}:= \phi_{\tilde{T},\tilde\alpha} \circ G^{-1}_T$.
	\begin{figure}[!hh]
		\centering
		\begin{tikzpicture}[scale=3]
		\draw[thick] (0,0) -- (0.5,0) -- (0,0.5) -- cycle;
		
		\draw[thick] (0,0) -- (0.5,0) -- (0,-0.5) -- cycle;
		
		\node[below left] at (0,0.1) {$O \, (0,0)$};
		\node[below] at (0.7,0) {$A \, (1,0)$};
		\node[above left] at (0,0.4) {$B \, (0,1)$};
		\node[below left] at (0,-0.4) {$D \, (0,-1)$};
		
		\put(1.39,22){$\tilde{T}_+$}
		\put(1.39,-26){$\tilde{T}_-$}
		\put(15,2){$\tilde{e}$}
		\end{tikzpicture}
		\caption{Two reference elements with vertices $O$, $A$, $B$, and $D$, and their coordinates. The edge $y=0$ is labeled as $\tilde{e}$.}
		\label{fig:reference-triangles}
	\end{figure}
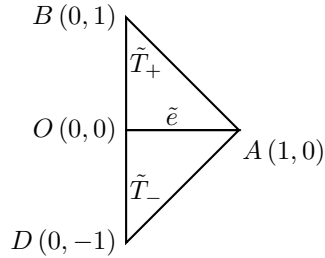
	For $\tilde{\cT}_{\tilde e} := \tilde{T}_+ \cup \tilde{T}_-$, we define $\phi_{\tilde{T},\tilde\alpha}: \tilde{\cT_e} \longrightarrow \mathbb{R}$ as
	$$
	\phi_{\tilde{T},\tilde\alpha}:= \tilde \psi_{\tilde\alpha} \Lambda^4_{\tilde{T},1} \Lambda^4_{\tilde{T},3}
	$$
	where $\Lambda_{\tilde{T},i}$ are both linear polynomials defined on $\tilde{T}$. 
	In particular, $\Lambda_{\tilde{T}_+,1}$ and $\Lambda_{\tilde{T}_-,1}$ vanish on the sides $OB$ and $OD$, respectively, 
	while $\Lambda_{\tilde{T},3}$ vanishes on $BA$ (for $\tilde{T}_{+}$) and $AD$ (for $\tilde{T}_{-}$). 
	For instance, we take $\Lambda_{\tilde{T}_+,1} := \tilde{x}$ 
	and $\Lambda_{\tilde{T}_+,3} := 1-\tilde{x}-\tilde{y}$;  one can similarly define 
	$\Lambda_{\tilde{T}_-,i}$. 
	We also introduce the continuous, piecewise affine 
	function $\tilde \psi_{\tilde\alpha}: \tilde{\cT}_{\tilde{e}} \longrightarrow \mathbb{R}$ that vanishes on the
	common edge $\tilde{e}$ and satisfies $\nabla \tilde \psi_{\tilde\alpha}|_{\tilde{T}_{\pm}} = \pm h_{\tilde{e}}^{-1} \tilde\alpha/2$. 
	Here, $\tilde\alpha$ is either $n_{\tilde e} := (0,1)$ (i.e.\ the unit normal vector to $\tilde e$ pointing from $\tilde{T}_-$ to $\tilde{T}_+$)
	or $\tau_{\tilde e} := (1,0)$ (i.e.\ the unit tangential vector parallel to $\tilde e$). 
	Hence, $\nabla \tilde \psi_{\tilde\alpha}$ has a discontinuity on $\tilde{e}$. We now 
	prove \eqref{bubbleproperties1}. For $e \in \partial \tilde{T}$ and $ e \neq \tilde{e}$, it is 
	clear from the definition that 
	$$
	\jump{\phi_{\tilde{e},\tilde\alpha}} := \phi_{\tilde{T},\tilde\alpha} = 0 
	\quad \text{and} \quad \avg{\phi_{{\tilde e},\tilde\alpha}} := \phi_{\tilde{T},\tilde\alpha} = 0 
	$$
	Using that $\phi_{\tilde{T},\tilde\alpha} = 0$ on $\tilde{e}$, the following clearly holds on $\tilde{e}$:
	\begin{align*}
	\sjump{\phi_{\tilde{e},\tilde\alpha}} := \phi_{\tilde{T}_+,\tilde\alpha} - \phi_{\tilde{T}_+,\tilde\alpha} = 0 \quad \text{and} \quad 
	\avg{ \phi_{\tilde{e},\tilde\alpha} } := \frac{\phi_{\tilde{T}_+,\tilde\alpha} +  \phi_{\tilde{T}_+,\tilde\alpha}}{2} = 0.
	\end{align*}
	Next, let $d_{\tilde{T}}:= \Lambda^4_{\tilde{T},1} \Lambda^4_{\tilde{T},3}$. For $e \in \partial \tilde{T}$, $ e\ne\tilde e$, 
	we note that by the product rule,
	\begin{align} \label{treff1}
	\avg{ \nabla \phi_{\tilde{e},\tilde\alpha} } := \nabla \phi_{\tilde{T},\tilde\alpha} = (d_{\tilde{T}} \nabla \tilde \psi_{\tilde\alpha} + \tilde \psi_{\tilde\alpha} \nabla d_{\tilde{T}}) = 0 .
	\end{align}
	In \eqref{treff1}, we have used that both $d_{\tilde T}$ and $\nabla d_{\tilde T}$ vanish 
	for any $e \ne \tilde e$. On $\tilde{e}$, we have that
	\begin{align*}  
	\avg{ \nabla \phi_{\tilde{e},\tilde\alpha} }  
	&:= \frac{(\nabla \phi_{\tilde{T}_+,\tilde\alpha})|_{\tilde{e}} + (\nabla \phi_{\tilde{T}_-,\tilde\alpha})|_{\tilde{e}}}{2} 
	= d_{\tilde{T}} \Big( \frac{h_{\tilde e}^{-1} \tilde\alpha  - h_{\tilde e}^{-1} \tilde\alpha}{4}\Big) = 0,
	\end{align*}
	since $\tilde \psi_{\tilde\alpha} = 0$ and $d_{\tilde{T}_+} = d_{\tilde{T}_-} = \tilde{x}^4 (1-\tilde x)^4$ on $\tilde e$. 
	This fully establishes \eqref{bubbleproperties1}. To prove \eqref{bubbleproperties2}, 
	we similarly have $\sjump{\nabla \phi_{\tilde{e},\tilde\alpha}}  = 0 \text{ on } \partial \tilde{T} \setminus \tilde{e}.$ 
	Using the properties of $\tilde \psi_{\tilde\alpha}$ and $d_{\tilde T}$ just mentioned, it follows that
	\begin{align*}
	\jump{\nabla \phi_{\tilde{e},\tilde\alpha}} \vert_{\tilde e} 
	& := (\nabla \phi_{\tilde{T}_+,\tilde\alpha})|_{\tilde e} - (\nabla \phi_{\tilde{T}_-,\tilde\alpha})|_{\tilde e }  \notag  \\ 
	& = d_{\tilde{T}_+} \nabla \tilde \psi_{\tilde\alpha} \vert_{\tilde T_+} - d_{\tilde{T}_-} \nabla \tilde \psi_{\tilde\alpha} \vert_{\tilde{T}_-}
	= h_{\tilde{e}}^{-1}  \tilde\alpha \Lambda^4_{\tilde{T},1} \Lambda^4_{\tilde{T},3},
	\end{align*}
	as desired.
\end{proof}
To prove another key property of $\phi_{\ehat,\alpha }$, we define an additional function $\zeta_{\alpha}: \cT_{\ehat}\to\RR$. Given 
$v_h \in \cV_h^k$, we define $\zeta_{\alpha}$ to equal  
$h_{\ehat }^{-1} \avg{ \partial_{\hat{n}}\nabla v_h}  \cdot \alpha$ on $\ehat$, where $\alpha \in \{\nhat,\tauhat\}$. We then extend $\zeta_{\alpha}$ to all of $\cT_{\ehat}$
by constants along
lines normal to $\ehat $. Next, we state a key lemma which will be used in the later analysis.
\begin{lemma} \label{bubblelifting}
	Let $L(\cdot)$ be the lifting operator defined in \eqref{eq:lifting} and  $\ehat \in \cC$ be any edge on the interface. 
        Then, for any $v_h \in \cV_h^k$, 
	$$
	\sum_{T \in \cT_h} \int_T L(v_h) : D^2 v \, dx = 0 
	$$
	where $v := \zeta_{\alpha} \phi_{\ehat,\alpha} \in \cV(0,0)$ and $\alpha \in \{\nhat, \tauhat\}$. 
\end{lemma}
\begin{proof}
By definition \eqref{eq:lifting}, we have
\begin{align*}
\sum_{T \in \cT_h} \int_T L(v_h) : D^2 v \, dx =  
-\sum_{e \in \tilde{\Gamma}_h} \int_e \jump{v_h} \avg{\partial_n \Delta v} \ ds
+ \sum_{e \in \tilde{\Gamma}_h\setminus \cC} \int_e \jump{\nabla v_h} \cdot \avg{\partial_n \nabla v} \ ds 
\end{align*}
Since $v$ is nonzero only on $\cT_{\ehat}$, it suffices to show that $\avg{\partial_n \nabla v}$ 
and $\avg{\partial_n \Delta v}$ vanish for $e \in \partial T_- \cup \partial T_+$. 
By construction, note that $\zeta_{\alpha}$ is smooth on $\cT_{\ehat}$. Using 
\eqref{bubbleproperties1}, it follows that 
\begin{equation}\label{eq:second_deriv_v}
\avg{\partial_n \nabla v }= \zeta_{\alpha} \, \avg{\partial_n \nabla \phi_{\ehat,\alpha}},
\end{equation}
and
\begin{align}
\avg{\partial_n \Delta v} = &
 \avg{\Delta \phi_{\ehat,\alpha}}\partial_n \zeta_{\alpha}
+ 2\avg{\partial_n \nabla \phi_{\ehat,\alpha}} \cdot \nabla \zeta_{\alpha} 
+ \avg{ \partial_n \Delta \phi_{\ehat,\alpha}} \zeta_{\alpha} \label{eq:third_deriv_v}
\end{align}
for any $e \in \partial T_- \cup \partial T_+$. 

Recall that for $T \in \{ T_+,T_-\}$, 
$\phi_{\ehat,\alpha}\vert_T = \psi_{\alpha} d_T$, where $d_T = \Lambda_{T,1}^4 \Lambda_{T,3}^4$ and 
$\psi_{\alpha}\vert_T$ is affine. Note that, by construction, $\avg{d_T} = d_T$ for any $e \in \partial T_{-} \cup \partial T_{+}$; 
the same holds true for all of the derivatives of $d_T$. 
Hence, by direct calculation,
\begin{align*}  
\avg{\partial_n \Delta \phi_{\ehat,\alpha}} &= \avg{\partial_n \psi_{\alpha}} \Delta d_T  
+ \avg{\psi_{\alpha}} \partial_n \Delta d_T + 2\partial_n \nabla d_T \cdot \avg{\nabla \psi_{\alpha}}. 
\end{align*}
Recall that $\Lambda_{T,1} \Lambda_{T,3}=0$ on any $e \ne \ehat$. 
Thus, by the chain rule, $d_T$ and 
all of its first, second, and third-order derivatives vanish for such edges. 
On $\ehat$, recall that $\avg{\psi_{\alpha}} = \psi_{\alpha} = 0$ and
$\avg{\nabla \psi_{\alpha}} = 0$ (cf.\ Eq.~\eqref{eq:avg_grad_psi}). 
Hence, the third term on the right-hand side of \eqref{eq:third_deriv_v} 
vanishes for all $ e\in \partial T_- \cup \partial T_+$. By analogous reasoning, 
the second-order derivative terms on the right-hand sides
of \eqref{eq:second_deriv_v} and \eqref{eq:third_deriv_v}
vanish as well for any $ e \in \partial T_{-} \cup \partial T_{+}$, giving the desired result. 
\end{proof}

\section{Reliability of the Error Estimator}
\label{sec:reliability}
We define the following error estimators: 
\begin{align}
\eta^2_1&:= \sum_{T   \in \cT_h } \|h_T^2 (f- \Delta^2 u_h)\|^2_{L^{2}(T)},
&&\eta^2_2:= \sum_{e \in \tilde{\Gamma}_h }\|h_e^{-\frac{3}{2}} \sjump{u_h}\|^2_{L^{2}(e)}, \nonumber \\
\eta^2_3&:=  \sum_{e \in \tilde{\Gamma}_h \setminus \cC}\|h_e^{-\frac{1}{2}} \sjump{ \nabla u_h}\|^2_{L^{2}(e)},
&&\eta^2_4:= \sum_{e \in \Gamma_h^{\rm int} } \|h_e^{\frac{1}{2}} \sjump{ \partial_n \nabla u_h}\|^2_{L^{2}(e)}, \nonumber \\
\eta^2_5&:=  \sum_{e \in \cC} \|h_e^{\frac{1}{2}} \avg{ \partial_n \nabla u_h} \|^2_{L^{2}(e)},  
&&\eta^2_6:=  \sum_{e \in \Gamma_h^{\rm int}} \|h_e^{\frac{3}{2}} \sjump{ \partial_n \Delta u_h} \|^2_{L^{2}(e)}. \label{eq:eta_defns}
\end{align}
The first estimator $\eta_1$ measures the standard element-wise PDE residual, while $\eta_2$, $\eta_3$, $\eta_4$, and $\eta_6$
measure the lack of $H^1$, $H^2$, $H^3$, and $H^4$ regularity of the discrete solution, respectively. 
The estimator $\eta_5$ measures
the extent to which the interface
condition 
$$
\partial_n \nabla u_h \big\vert_{\Omega_j} = 0, \qquad j \in \{1,2\}
$$
holds along the fold $\cC$. Please note that in making these definitions, we 
have modified the definition of the jumps in $u_h$ and $\nabla u_h$
along Dirichlet boundary edges; for any $e \in  \Gamma_h^D$ belonging 
to $T \in \triang$, we set
$$
\jump{u_h} := g - u_h \vert_T \quad \text{and} \quad
\jump{\nabla u_h} := \Phi - \nabla u_h \vert_T. 
$$
In this case, the linear form $l_h$ in the discrete problem (cf.\ Eq.~\eqref{eq:discrete_linear}) will no longer
contain the terms involving $g$ and $\Phi$. Note that only $\eta_2$ and $\eta_3$ are
affected by this modification.

The next theorem ensures the reliability of the error estimators $\eta_i$, $1 \le i \le 6$.
\begin{theorem}
	Let $u \in \mathcal{V}(g,\Phi)$ be the solution to \eqref{eq:variational_form},
	and let $u_h \in \mathcal{V}^k_h$ be a solution of the discrete problem \eqref{eq:discrete_problem}. 
	Then the following reliability estimate holds:
	\begin{align*}
	\|u-u_h\|^2_{DG} \lesssim \sum_{i=1}^{6} \eta_i^2.
	\end{align*}
\end{theorem}
\begin{proof}
	We refer to \cite[Theorem 3.3]{tschernernumerical} for the proof.
\end{proof}

\section{Efficiency of the Error Estimator}
\label{sec:efficiency}
Next, we prove discrete local efficiency estimates. The main result is as follows:
\begin{theorem} \label{mainresult:eff}
Let $u \in \mathcal{V}(g,\Phi)$ satisfy \eqref{eq:variational_form}, $v_h \in \cV_h^k$ and $\bar{f} \in \mathbb{P}_0(T)$ be a piecewise constant approximation of $f$. Then, it holds that
		\begin{align}
	\sum_{T   \in \cT_h }\|h_T^2 (f- \Delta^2 v_h)\|^2_{L^{2}(T)} & \lesssim  \sum_{T   \in \cT_h} \bigg(|u-v_h|^2_{2,T}  +  h_T^2\|f- \bar{f}\|^2_{L^2(T)} \bigg), \label{efficiency1} \\ 
	 \sum_{e \in \tilde{\Gamma}_h }\|h_e^{-\frac{3}{2}} \sjump{v_h}\|^2_{L^{2}(e)} 
	 & \lesssim  \|u-v_h\|^2_{DG}, \label{efficiency2} \\
	   \sum_{e \in \tilde{\Gamma}_h \setminus \cC}\|h_e^{-\frac{1}{2}} \sjump{ \nabla v_h}\|^2_{L^{2}(e)} &\lesssim  \|u-v_h\|^2_{DG}, \label{efficiency3} \\
	    \sum_{e \in \Gamma_h^{\rm int} } \|h_e^{\frac{1}{2}} \sjump{ \partial_n \nabla v_h}\|^2_{L^{2}(e)} 
	    & \lesssim \sum_{T   \in \cT_h} \bigg(|u-v_h|^2_{2,T}  +  h_T^2\|f- \bar{f}\|^2_{L^2(T)} \bigg), \label{efficiency4} \\
	    \sum_{e \in \cC} \|h_e^{\frac{1}{2}} \avg{ \partial_n \nabla v_h} \|^2_{L^{2}(e)}
	    & \lesssim \sum_{T   \in \cT_h} \bigg(|u-v_h|^2_{2,T}  +  h_T^2\|f- \bar{f}\|^2_{L^2(T)} \bigg), \label{efficiency5} \\
	    \sum_{e \in \Gamma_h^{\rm int}} \|h_e^{\frac{3}{2}} \sjump{ \partial_n \Delta v_h} \|^2_{L^{2}(e)} & \lesssim  \sum_{T   \in \cT_h} \bigg(|u-v_h|^2_{2,T}  +  h_T^2\|f- \bar{f}\|^2_{L^2(T)} \bigg). \label{efficiency6}
	\end{align}
\end{theorem}
Note that the estimates \eqref{efficiency2} and \eqref{efficiency3} follow immediately from the
definition of the DG norm \eqref{eq:dg_norm1} and the fact that $u \in H^1(\Omega) \cap H^2(\Omega_1 \cup \Omega_2)$. 
The novel estimate is \eqref{efficiency5}, and hence we focus on its proof. 
Since the estimates \eqref{efficiency1}, \eqref{efficiency4} and \eqref{efficiency6} have been previously 
shown, we provide references for the proofs but omit the details. 
\begin{proof}
The proof of \eqref{efficiency1} is similar to the proof of \cite[Lemma 4.1]{MR2684360}, while 
the proof of \eqref{efficiency4} is based on standard ``edge"  bubble function techniques, and we refer to \cite[Lemma 4.1]{MR2684360}  for the proof.
Estimate \eqref{efficiency6} can be shown by following \cite[Lemma 4.2]{MR2684360}.
To prove \eqref{efficiency5}, it is sufficient to show that
	\begin{align*}  
\|h_{\ehat }^{\frac{1}{2}} \avg{ \partial_{\hat{n}} \nabla v_h }\|^2_{L^2(\ehat )}  \lesssim \sum_{T   \in \cT_{\ehat }} \bigg( |u-v_h|^2_{2,T}  + h_T^2\|f- \bar{f}\|^2_{L^2(T)} \bigg), 
\end{align*}
%
%
%
%
%
\par
\noindent
where $v_h \in \cV_h^k$,  $\cT_{\ehat } : = T_{-}\cup T_{+}$ and $\ehat  \in \cC$ is shared by two 
neighboring triangles $T_{-},T_{+} \in \cT_h$, where $T_{-} \subset \Omega_1$ and $T_{+} \subset \Omega_2$. 
As depicted in Figure \ref{fig:rhombus}, let $\hat{n}$ denote the outer unit normal vector to $T_{+}$. 
Let $v \in \cV(0,0)$, and for $v_h \in \mathcal{V}^k_h$, define $\varepsilon := u -v_h$. 
Using regularity of $u$ and $v$ and definition of lifting operator $L(\cdot)$ (Eq.~\eqref{eq:lifting}), it holds that $L(u)=L(v)=0$. 
By \eqref{eq:discrete}, we also have that $a_h(u,v)=l(v) $. Let us consider
\begin{align*}
a_h(\varepsilon,v) &= a_h(u,v) -  a_h(v_h,v) \notag \\
& = l(v)  -a_h(v_h, v).     
\end{align*}
Using the definitions of $a_h(\cdot, \cdot)$, $l(\cdot)$, $L(\cdot)$ and elementwise integration by parts (cf.~\cite[Theorem 3.3]{tschernernumerical}), we arrive at
\begin{align}
a_h(\varepsilon,v)  &= \sum_{T \in \mathcal{T}_h} \int_T \big[(f- \Delta^2 v_h) v  - L(v_h) : D^2 v \big]dx
+  \sum_{e \in \Gamma^{\rm int}_h} \int_{e} \avg{ \nabla v} \cdot \sjump{\partial_n \nabla v_h} ds  \notag \\ 
&  \quad + \sum_{e \in \mathcal{C}} \int_{e} \avg{ \partial_n \nabla v_h} \cdot \sjump{\nabla v} ds 
- \sum_{e \in \Gamma^{\rm int}_h}  \int_{e} \avg{ v } \sjump{\partial_n \Delta v_h} ds  \notag \\ 
&  \quad -  \sum_{e \in \tilde{\Gamma}_h} \int_{e} \frac{\gamma_0}{h^3} \sjump{v_h} \sjump{v} ds 
- \sum_{e \in \tilde{\Gamma}_h \setminus \mathcal{C}} \int_{e} \frac{\gamma_1}{h} \sjump{\nabla v_h} \cdot \sjump{\nabla v} ds. \label{est2}
\end{align}
%
%
%
Next, the idea is to consider a test function $v$ in such a way that all but the first and fourth terms vanish 
on the right-hand side of \eqref{est2}. 
For $\ehat  \in \partial T_- \cap \partial T_+$ and $T \in \{T_-, T_+\}$, recall the bubble 
function $\phi_{\ehat,\alpha } : \Omega \longrightarrow \mathbb{R}$, where $\alpha \in \{\nhat, \tauhat\}$ (defined in \eqref{novelbubble}; 
see Figure \ref{fig:rhombus}). Consider 
 the test function $v = \zeta_{\tauhat} \phi_{\ehat,\tauhat }+\zeta_{\nhat} \phi_{\ehat,\nhat } \in \cV(0,0)$ (similar to Lemma \ref{bubblelifting}). 
In the following, we assume that $\avg{\partial_{\hat n} \nabla v_h}$ is nonzero on some subset of $\ehat$ of positive measure; 
otherwise, the contribution from $\ehat$ to \eqref{efficiency5} is zero. 
Using this assumption, the properties of $\phi_{\ehat,\alpha }$ (Lemma \ref{bubble:properties}), continuity of $\zeta_{\alpha}$, and 
Lemma \ref{bubblelifting}, Eq.~\eqref{est2} reduces to
\begin{align} \label{effp1}
\int_{\ehat }  \avg{ \partial_{\hat{n}} \nabla v_h} \cdot  \sjump{\nabla v} ds = \int_{\cT_{\ehat }} D^2 \varepsilon : D^2 v ~dx - \int_{\cT_{\ehat }} (f- \Delta^2 v_h) v ~dx.
\end{align}
Using the structure of $ \Lambda_{T,i}$ for $i \in \{1,3\}$ and the standard 
norm equivalence on finite dimensional space $\cV_h^k$ \cite[Lemma 4.22]{MR4793681},
we have
\begin{align}
\int_{\ehat } \avg{ \partial_{\hat{n}} \nabla v_h} \cdot  \sjump{\nabla v } ds 
&= \int_{\ehat } h_{\ehat }^{-2} \Lambda^4_{T,1}  \Lambda^4_{T,3} 
\big(\avg{ \partial_{\nhat} \partial_{\nhat} v_h}^2 +  \avg{ \partial_{\nhat} \partial_{\tauhat} v_h}^2\big)~ds \notag \\ 
& \gtrsim \|h_{\ehat }^{-1} \avg{ \partial_{\nhat} \nabla v_h} \|^2_{L^{2}(\ehat )}. \label{effp2}
\end{align}
By inserting \eqref{effp2} in \eqref{effp1}, we get
\begin{align} 
\|h_{\ehat }^{-1} \avg{ \partial_{\hat{n}} \nabla v_h} \|^2_{L^{2}(\ehat )} & \lesssim  \int_{\cT_{\ehat }} D^2 \varepsilon : D^2 v ~dx - \int_{\cT_{\ehat }} (f- \Delta^2 v_h) v ~dx. \nonumber
\end{align}
By using Cauchy Schwarz inequality and standard inverse inequality \cite[Theorem 3.2.6]{MR0520174} on $v$, we further deduce
\begin{align}
\|h_{\ehat }^{-1} \avg{ \partial_{\hat{n}} \nabla v_h} \|^2_{L^{2}(\ehat )} & \lesssim \|h_T^{\frac{1}{2}} (f- \Delta^2 v_h)\|_{L^{2}(\cT_{\ehat })} \|h_T^{-\frac{1}{2}} v \|_{L^{2}(\cT_{\ehat })} \nonumber \\ &  \qquad \qquad  \qquad +  h_T^{-\frac{3}{2}}\|D^2 \varepsilon\|_{L^{2}(\cT_{\ehat })}  h_T^{\frac{3}{2}}\|D^2 v\|_{L^{2}(\cT_{\ehat })}, \nonumber \\
\implies \|h_{\ehat }^{-1} \avg{ \partial_{\hat{n}} \nabla v_h} \|^2_{L^{2}(\ehat )}  
& \lesssim \bigg( \|h_T^{\frac{1}{2}} (f- \Delta^2 v_h)\|_{L^{2}(\cT_{\ehat })} \nonumber \\ & \qquad \qquad \qquad +  h_T^{-\frac{3}{2}}\|D^2 \varepsilon\|_{L^{2}(\cT_{\ehat })}  \bigg) \|h_T^{-\frac{1}{2}} v \|_{L^{2}(\cT_{\ehat })}. \label{effp4}
\end{align}
Using the definition of $v$ and similar ideas as in \cite[Theorem 3.2]{MR2034620}, it holds that
\begin{align} \label{effp3}
\|v\|_{L^{2}(\cT_{\ehat })} \lesssim \|h_{\ehat }^{-\frac{1}{2}} \avg{ \partial_{\hat{n}} \nabla v_h} \|_{L^{2}(\ehat )}.
\end{align}
We use \eqref{efficiency1} and insert \eqref{effp3} in \eqref{effp4} to obtain 
\begin{align*}
 \|h_{\ehat }^{\frac{1}{2}} \avg{ \partial_{\hat{n}} \nabla v_h } \|^2_{L^{2}(\ehat )}  \lesssim  \sum_{T   \in \cT_{\ehat }} \bigg( |u-v_h|^2_{2,T}  + h_T^2\|f- \bar{f}\|^2_{L^2(T)} \bigg),
\end{align*}
as desired.

\end{proof}

\section{A priori error estimates}
\label{sec:medius}
In this section, we derive improved \emph{a priori} error bounds under 
minimal regularity assumptions on the exact solution $u$ of 
\eqref{eq:variational_form}, motivated by the celebrated analysis of Gudi \cite{MR2684360}. The main result of this section is as follows:
\begin{theorem}
	Let $u \in \cV(g,\Phi)$ and $u_h \in \cV^k_h$ be the solutions 
to \eqref{eq:variational_form} and \eqref{eq:discrete_problem}, respectively. Then, the following a priori error estimate holds:
	\begin{align} \label{proof:apriori}
	\|u-u_h\|_{DG} \lesssim \bigg( \inf_{v_h \in \mathcal{V}^k_h} \|u-v_h\|_{DG} +\Big(\sum_{T   \in \cT_h} h_T^2\|f- \bar{f}\|^2_{L^2(T)} \Big)^{\frac{1}{2}} \bigg).
	\end{align}   
\end{theorem}
\begin{proof}
	Let $v_h \in \mathcal{V}^k_h$ with $v_h \neq u_h$. 
By the triangle inequality and coercivity of $a_h(\cdot,\cdot)$, we can derive the standard estimate
\begin{equation}\label{eq:starting_point}
\| u- u_h\|_{DG} \lesssim \| u - v_h\|_{DG} 
+ \sup_{w_h \in \cV_h^k\setminus\{0\}} \frac{a_h(u_h-v_h, w_h)}{\| w_h \|_{DG}}.
\end{equation}
We focus on bounding the numerator in the second term on the right-hand side of \eqref{eq:starting_point}. 
	Using the definition of the enriching operator $E_h$ (cf.\ Subsection \ref{subsec:enriching}), 
we have that for any $\psi \in \cV_h^k$, $a(u, E_h \psi)= a_h(u, E_h \psi )$. 
Taking $\psi:= u_h-v_h \in \mathcal{V}^k_h$ and using both the variational form \eqref{eq:variational_form} 
and stability of $E_h$ (Eq.~\eqref{estimate:stabilityenriching}), it holds that
	\begin{align}
	a_h (u_h - v_h, \psi ) 
	& = l_h(\psi) - l(E_h \psi) + a(u,E_h \psi)- a_h ( v_h, \psi ) \notag \\ 
	& =  a_h(u-v_h, E_h \psi) + l_h(\psi) - l(E_h \psi) - a_h(v_h, \psi - E_h \psi ) \notag \\ 
	& \lesssim \|u-v_h\|_{DG} \|E_h \psi\|_{DG} +  l_h(\psi) - l(E_h \psi) - a_h(v_h, \psi - E_h \psi )  \notag \\ 
	& \lesssim \|u-v_h\|_{DG} \|\psi\|_{DG} +  l_h(\psi) - l(E_h \psi) - a_h(v_h, \psi - E_h \psi ). \label{eq10}
	\end{align}
Setting $\chi:= \psi - E_h \psi$ and using \cite[Eq.~(3.51)]{tschernernumerical}, it holds that
	\begin{align} \label{eq2}
	l_h&(\psi) - l(E_h \psi)  -a_h (v_h, \chi) 
	=  \underbrace{\sum_{T \in \mathcal{T}_h} \int_T \big[(f- \Delta^2 v_h) \chi - L(v_h) : D^2 \chi \big]dx}_{\rm (I):= }  \notag \\ 
	& \quad  -   \underbrace{\sum_{e \in \Gamma^{\rm int}_h}  \int_{e} \avg{ \chi} \sjump{\partial_n \Delta v_h} ds}_{\rm (II) :=}   
+  \underbrace{\sum_{e \in \cC} \int_{e} \sjump{\nabla \chi} \cdot \avg{\partial_n \nabla v_h} ds}_{\rm (III):= }  \\  
	& \hspace*{-0.2cm} +  \underbrace{\sum_{e \in \Gamma^{\rm int}_h} \int_{e}\avg{\nabla \chi} \cdot \sjump{\partial_n \nabla v_h} ds}_{\rm (IV) := }   
-  \underbrace{\sum_{e \in \tilde{\Gamma}_h} \int_{e} \frac{\gamma_0}{h^3} \sjump{v_h} \sjump{\chi} ds}_{\rm (V):=}   \notag
\quad -  \underbrace{\sum_{e \in \tilde{\Gamma}_h \setminus \cC} \int_{e} \frac{\gamma_1}{h} \sjump{\nabla v_h} \cdot \sjump{\nabla \chi} ds}_{\rm (VI):= }. 
	\end{align}
The idea from here is to bound each term on the right-hand side of \eqref{eq2} 
by $\|\psi\|_{DG}$ times one of the discrete local estimators defined in Theorem \ref{mainresult:eff}. 

Using the stability of the lifting operator $L$ (Eq.~\eqref{stabilityL1}), 
properties of $E_h$ (Lemma \ref{estimate:enriching}), and the Cauchy Schwarz inequality, 
we have
	\begin{align}
 \textrm{(I)}
	\lesssim \bigg( \sum_{T \in \mathcal{T}_h}  \|h_T^2 &(f - \Delta^2 v_h)\|^2_{L^2(T)} \notag \\ 
	& +   \sum_{e \in \tilde{\Gamma}_h}  \big \|\sqrt{\frac{\gamma_0}{h_e^3}} \sjump{v_h}\big \|^2_{L^2(e)}  + \sum_{e \in \tilde{\Gamma}_h \setminus \cC} \big \|\sqrt{\frac{\gamma_1}{h_e}} \sjump{ \nabla v_h} \big \|^2_{L^2(e)} \bigg)^{\frac{1}{2}} \|\psi\|_{DG}. \label{eq3}
	\end{align} 
	
Next, by the Cauchy Schwarz inequality, 
	\begin{align}
        \textrm{(II)}
	& \lesssim  \bigg ( \sum_{e \in \Gamma^{\rm int}_h} \|h_e^{-\frac{3}{2}} \avg{ \chi}\|^2_{L^2(e)} \bigg)^{\frac{1}{2}}  
\bigg ( \sum_{e \in \Gamma^{\rm int}_h} \|h_e^{\frac{3}{2}} \sjump{\partial_n \Delta v_h}  \|^2_{L^2(e)}\bigg)^{\frac{1}{2}}.   \label{eq6} 
	\end{align}
Using the trace theorem with scaling \cite[Eq.~2.8]{MR2034620} and 
Lemma \ref{estimate:enriching} (with both $\beta = 0$ and $\beta=1$), we bound the first term on right-hand side of \eqref{eq6} as 
	\begin{align}
	\sum_{e \in \Gamma^{\rm int}_h} \|h_e^{-\frac{3}{2}} \avg{ \chi} \|^2_{L^2(e)} 
	&  \lesssim \sum_{T   \in \cT_h} h^{-3}_T \|\chi\|^2_{\partial T} \notag \\ 
	& \lesssim \sum_{T   \in \cT_h} h^{-3}_T \big(h^{-1}_T \|\chi\|_{L^2(T)}^2 + h_T |\chi|^2_{1,T} \big) \lesssim \| \psi\|^2_{DG}. \label{eq5} 
	\end{align}
	Combining \eqref{eq5} with \eqref{eq6} implies
	\begin{align}
	-  \sum_{e \in \Gamma^{\rm int}_h}  \int_{e} \avg{ \chi} \sjump{\partial_n \Delta v_h} ds \lesssim \bigg (\sum_{e \in \Gamma^{\rm int}_h} \|h_e^{\frac{3}{2}} \sjump{\partial_n \Delta v_h}  \|_{L^2(e)}\bigg)^{\frac{1}{2}} \|\psi\|_{DG}.  \label{eq4} 
	\end{align}
The remaining terms on the right-hand side of \eqref{eq2} can be bounded with nearly identical arguments; for example, 
Lemma \ref{estimate:enriching} with $\beta = 1$ and $\beta=2$ will be used for (III), (IV), and (VI), 
all of which involve jumps and averages of $\chi$ and $\nabla \chi$. 
	In particular, we have:
	\begin{align}
\textrm{(III)} &
\lesssim \bigg ( \sum_{e \in \cC}\|h_e^{\frac{1}{2}} \avg{ \partial_n \nabla v_h } \|^2_{L^{2}(e)}   \bigg)^{\frac{1}{2}}  \|\psi\|_{DG}   \label{eq7}   \\ 
\textrm{(IV)} &
\lesssim \bigg ( \sum_{e \in \Gamma^{\rm int}_h} \|h_e^{\frac{1}{2}} \sjump{ \partial_n \nabla v_h}\|^2_{L^{2}(e)}  \bigg)^{\frac{1}{2}} \|\psi\|_{DG}  \label{eq8}  \\ 
\textrm{(V)} + \textrm{(VI)} 
&  \lesssim \bigg (   \sum_{e \in \tilde{\Gamma}_h} \|h_e^{-\frac{3}{2}} \sjump{v_h}\|^2_{L^{2}(e)}
 +  \sum_{e \in \tilde{\Gamma}_h \setminus \cC }\|h_e^{-\frac{1}{2}} \sjump{ \nabla v_h}\|^2_{L^{2}(e)}   \bigg)^{\frac{1}{2}} \|\psi\|_{DG}. \label{eq9} 
	\end{align}
By combining \eqref{eq10} with \eqref{eq2}, \eqref{eq3}, \eqref{eq4}, \eqref{eq7}, \eqref{eq8} and \eqref{eq9}, we arrive at  
	\begin{align} \label{eq:final_point}
		a_h (u_h - v_h, \psi ) &  \lesssim \bigg( \|u-v_h\|_{DG}  + \Big[ \sum_{T \in \mathcal{T}_h}  \|h_T^2 (f - \Delta^2 v_h)\|^2_{L^2(T)} \nonumber 
		\\ & \quad  +  \sum_{e \in \tilde{\Gamma}_h}  \big \|\sqrt{\frac{\gamma_0}{h_e^3}} \sjump{v_h}\big \|^2_{L^2(e)}  + \sum_{e \in \tilde{\Gamma}_h \setminus \cC} \big \|\sqrt{\frac{\gamma_1}{h_e}} \sjump{ \nabla v_h} \big \|^2_{L^2(e)} \nonumber  
		 \\ & \quad +  \sum_{e \in \Gamma^{\rm int}_h} \|h_e^{\frac{1}{2}} \sjump{ \partial_n \nabla v_h}\|^2_{L^{2}(e)} + \sum_{e \in \cC}\|h_e^{\frac{1}{2}} \avg{ \partial_n \nabla v_h } \|^2_{L^{2}(e)}  \nonumber 
		  \\ & \quad + \sum_{e \in \Gamma^{\rm int}_h} \|h_e^{\frac{3}{2}} \sjump{\partial_n \Delta v_h}  \|_{L^2(e)} \Big]^{\frac{1}{2}} \bigg) \|\psi\|_{DG}.
	\end{align}
	Substituting \eqref{eq:final_point} in \eqref{eq:starting_point} and employing the 
discrete local efficiency estimates from Theorem \ref{mainresult:eff} and yields the desired estimate \eqref{proof:apriori}.
\end{proof}

\section{Numerical Examples}
\label{sec:numerical_examples}
We now illustrate the performance of the estimators described above with
an adaptive algorithm implemented with the \texttt{deal.II} finite element library \cite{arndt2022deal}. 
The library's tutorial \cite{zhang_2021_5812174} on the 
application of a symmetric interior penalty (SIP)DG method to a Poisson problem 
 served as starting point for the implementation here. 
Note that \texttt{deal.II} uses quadrilateral mesh elements, in contrast to the 
setting for the theoretical results presented above. The visualizations are obtained
with \texttt{VisIt} \cite{Childs_High_Performance_Visualization--Enabling_2012}. 

For all simulation cases, we use 
second order elements $\cV_h^2$ (Eq.~\eqref{eq:dg_fe_space}) 
and solve the discrete linear systems with a sparse
direct solver. The library currently only supports derivatives of 
finite element basis functions up to the third order, meaning we cannot report results
for $\eta_1$.\footnote{Note that the biharmonic operator applied to a Q2 finite element
function is not generally equal to zero.} Accordingly, we redefine for the results below: 
\begin{equation}\label{eq:eta_tot_redefine}
\eta_{\rm tot} := \left( \eta_2^2 + \eta_3^2 + \eta_4^2 + \eta_5^2 + \eta_6^2\right)^{1/2}, 
\end{equation}
where each $\eta_i$ is defined by \eqref{eq:eta_defns}.  
The adaptive refinement is implemented using a standard ``solve-estimate-mark-refine'' loop
with refinement fraction $\theta = 0.1$.  

We present three numerical examples. The first two feature (piecewise) linear folds, 
which means that the fitted geometry assumption ($\cC = \cC_h$) is satisfied; the
third example features a sinusoidal fold, for which the assumption is violated. 
As discussed in Section \ref{sec:introduction}, in theory, this leads to a so-called geometric
consistency error \cite[Theorem 4.3]{MR4699572}. 
 Consistent with the results reported in \cite{MR4699572,tschernernumerical}, 
however, the error is not detected in our simulations. 

For the first numerical example, an analytic solution to the folding model \eqref{eq:variational_form}
is available, allowing us to calculate the simulation error in the DG norm \eqref{eq:dg_norm1}. 
In the latter two examples, however, no such analytic solution is known. 
In an analogous situation, 
the authors in \cite{MR4699572,tschernernumerical} approximated 
the simulation error using Aitken extrapolation and Galerkin orthogonality.
Galerkin orthogonality, however, requires that the true solution $u$ is  $H^4(\Omega_1 \cup \Omega_2) \cap H^1(\Omega)$ regular, 
which is not expected to hold for the 
second and third numerical examples presented below, owing to the combination of the domain $\Omega$, the fold $\cC$, and 
the prescribed Dirichlet boundary conditions $g$ and $\Phi$. 
Hence, for these cases we only report the individual and total estimators for  
uniform and adaptive mesh refinement. 

Finally, we mention that for our experiments, 
the penalty parameter values $\gamma_0 = \gamma_1 = 10$
that were reported in 
\cite{MR2755946,MR4699572,tschernernumerical}
were not sufficiently large to ensure consistent convergence results throughout 
our adaptive mesh cycles. 
%
In practice, starting from $\gamma_0 = \gamma_1 = 10$, 
we determine the penalty parameter values by increasing them both
by an integer factor until the results were stable for every mesh refinement. 
In particular, we use $\gamma_0 = \gamma_1 = 30$ for the first numerical example, 
$\gamma_0 = \gamma_1 = 70$ for the second example, and 
$\gamma_0 = \gamma_1 = 50$ for the third example. 

\subsection{Example 1: Flat fold}
First consider $\Omega = (0,1)^2$ and a constant fold parameterized as 
$\cC(x_2) = (1/2,  x_2)$, where $x_2 \in(0,1)$. 
In this case, one can construct a one-dimensional solution to the continuous folding 
problem \eqref{eq:variational_form} as 
$$
 u(x_1,x_2) = \begin{cases} 
0, \qquad x_1 \in [0,1/2)  \\
\big[\frac12 (x_1-1/2)^3 - (x_1-1/2)^2 + (x_1-1/2)\big] e^{x_1-1/2}, \qquad x_1 \in [1/2,1],
\end{cases}
$$
which is $C^{\infty}(\Omega_1 \cup \Omega_2)$ and 
satisfies the interface conditions \eqref{eq:interface_conditions}. 
Using the method of manufactured solutions, we then set
$f(x_1,x_2) := \Delta^2 u(x_1,x_2)$ for $(x_1,x_2) \in \Omega$. We take Dirichlet 
boundary conditions everywhere (so that $\partial_D \Omega = \partial \Omega$) 
with $g = u\vert_{\partial \Omega}$, and 
$\Phi = \nabla g$.    
We plot the adaptive mesh at level 12 and 24 in Figure \ref{fig:example1_visuals}, (left and middle, respectively) 
as well as the numerical solution (right). 
We observe the strongest refinement near the nonzero Dirichlet boundaries; in this case, 
the analytic solution in the neighborhood of the fold at $x_1 = 1/2$ is either exactly linear
(for $x_1 < 1/2$) or linear to an extremely good approximation (for $x_1 > 1/2$). Hence, 
the estimator $\eta_5$ (cf.\ Eq.~\eqref{eq:eta_defns}) that determines refinement near 
the fold $\cC$ is relatively small (cf.~Figure \ref{fig:example1_estimators}, right), explaining the lack 
of strong refinement near $\cC$. 

The convergence behavior of the error $\|u-u_h\|_{DG}$ and 
the estimator $\eta_{\rm tot}$ is depicted in Figure \ref{fig:example1_estimators} (left), illustrating 
the optimal convergence rate with respect to the degrees of freedom. This figure also confirms
the reliability of the error estimator. The average efficiency index (defined to be the ratio between 
the estimator and the DG norm error) across refinement levels is approximately 2.4, which particularly validates our findings in Theorem \ref{mainresult:eff}.
The convergence plot for each individual error 
estimator $\eta_i$, $2 \leq i \leq 6$, is also depicted in Figure \ref{fig:example1_estimators} (right). 

\begin{figure}
   \begin{center}
       \includegraphics[width=0.30\textwidth]{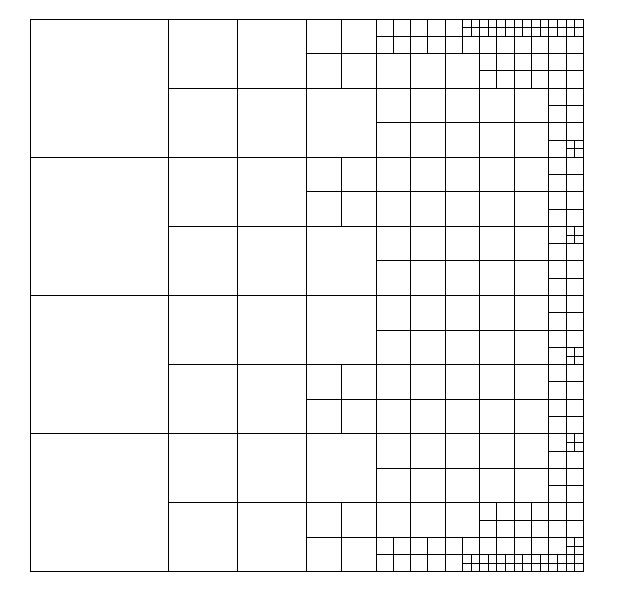}
\hskip-0.75em
       \includegraphics[width=0.30\textwidth]{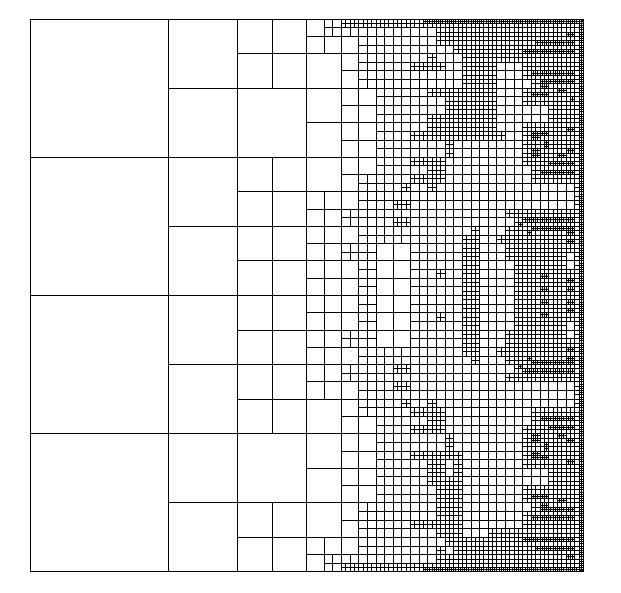}
       \includegraphics[width=0.36\textwidth]{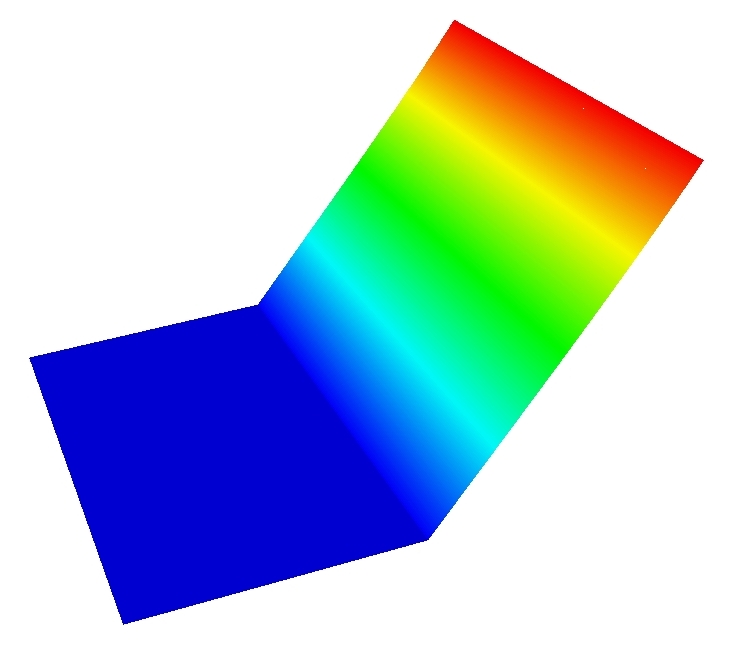} 
      \hskip-2em \includegraphics[width=0.08\textwidth]{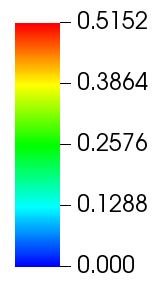} 
   \end{center}
\caption{
Adaptively refined mesh at level 12 and 24 (left and middle, respectively) and folding pattern (right) for Example 1.}
\label{fig:example1_visuals}
\end{figure}
\begin{figure}
   \begin{center}
       \includegraphics[width=0.49\textwidth]{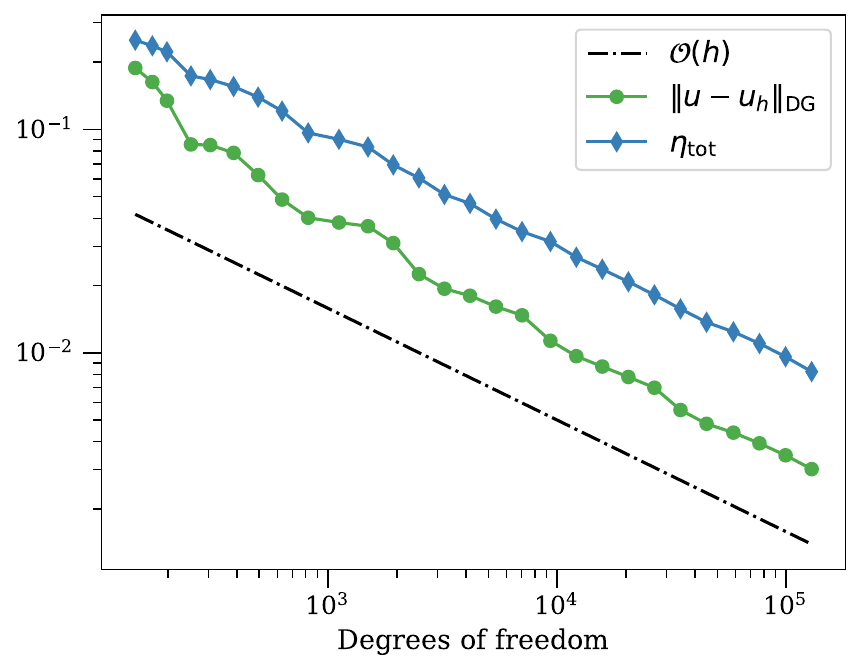}
       \includegraphics[width=0.49\textwidth]{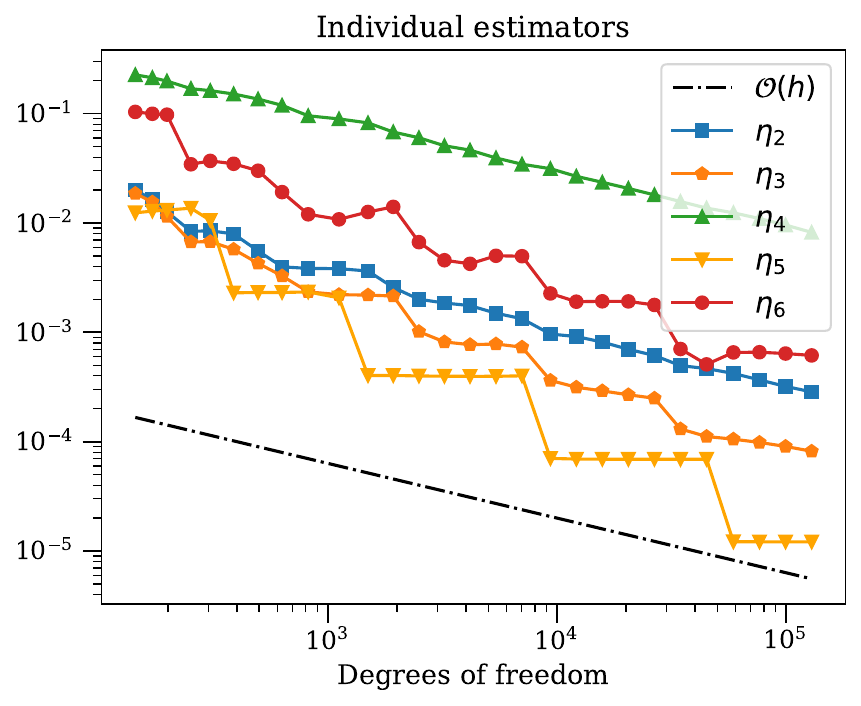}
   \end{center}
\caption{
Example 1: error in the DG norm \eqref{eq:dg_norm1} and 
the total error estimator \eqref{eq:eta_tot_redefine} (left);  individual 
estimators, as defined by \eqref{eq:eta_defns} (right). 
}
\label{fig:example1_estimators}
\end{figure}

\subsection{Example 2: Flapping mechanism with ``V-shaped'' fold}
Consider next $\Omega = (0,1)^2$ and the piecewise linear, ``V-shaped'' fold parameterized as 
$\cC(x_1) = (x_1,  \frac12 ( 1 + |x_1- \frac12 | ))$, where $ x_1 \in(0,1)$. In this case, we
consider mixed (i.e.\ both Dirichlet and natural) boundary conditions intended to mimic 
the ``flapping'' mechanism that can occur when two corners of a prepared elastic sheet are 
compressed (cf.\ \cite[Section 5.2]{bartels2022modeling}). In particular, 
we consider the Dirichlet data $g(x_1,x_2) = 0.35\sin(\pi x_1)$ and $\Phi = \nabla g$ at the boundary $x_2=1$. 
At the point $(x_1,x_2) = (\frac12,0)$, we prescribe $u = 1$; at all other boundary points, 
we consider natural boundary conditions \eqref{eq:natty_bcs}. The forcing function $f(x_1,x_2) = 0$.
We note that the exact solution $u$ is not known in this example. 
Two adaptively refined meshes (left and middle) are shown in Figure \ref{fig:example2_visuals}, 
as is the corresponding folding pattern (right). Notice that the refinement concentrates around the low regularity regions, namely 
near the ``tip'' of the crease $(x_1,x_2)=(\frac12, \frac12)$, 
as well as near the point $(x_1,x_2)=(\frac12, 0)$, where the solution ``pinned''.

The convergence behaviour of the error estimator $\eta_{\rm tot}$ 
is shown in Figure \ref{fig:example2_estimators} (left) for both adaptive 
and uniform mesh refinement;   
the convergence rates in the former case are nearly optimal, while 
those in the latter are suboptimal. The convergence plot for 
each individual estimator $\eta_i$, $2 \le i \le 6$, is also 
depicted in Figure \ref{fig:example2_estimators} (right).  
\begin{figure}
   \begin{center}
       \includegraphics[width=0.28\textwidth]{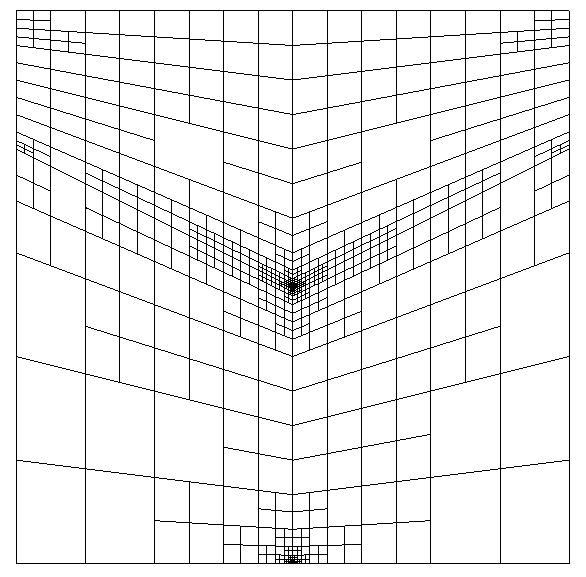}
\hskip-0.25em
       \includegraphics[width=0.28\textwidth]{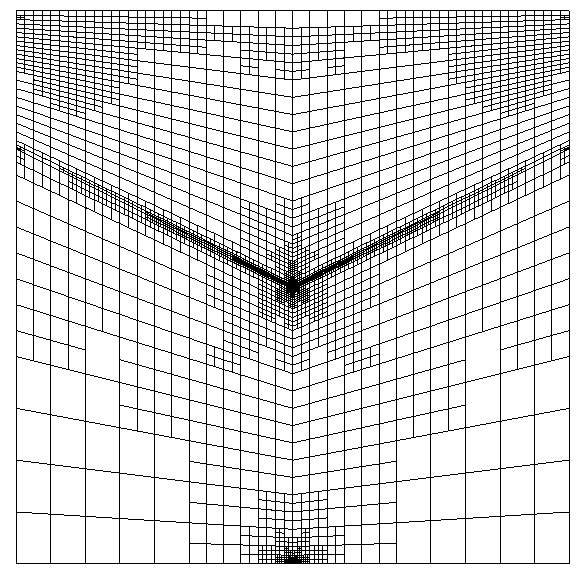}
       \includegraphics[width=0.34\textwidth]{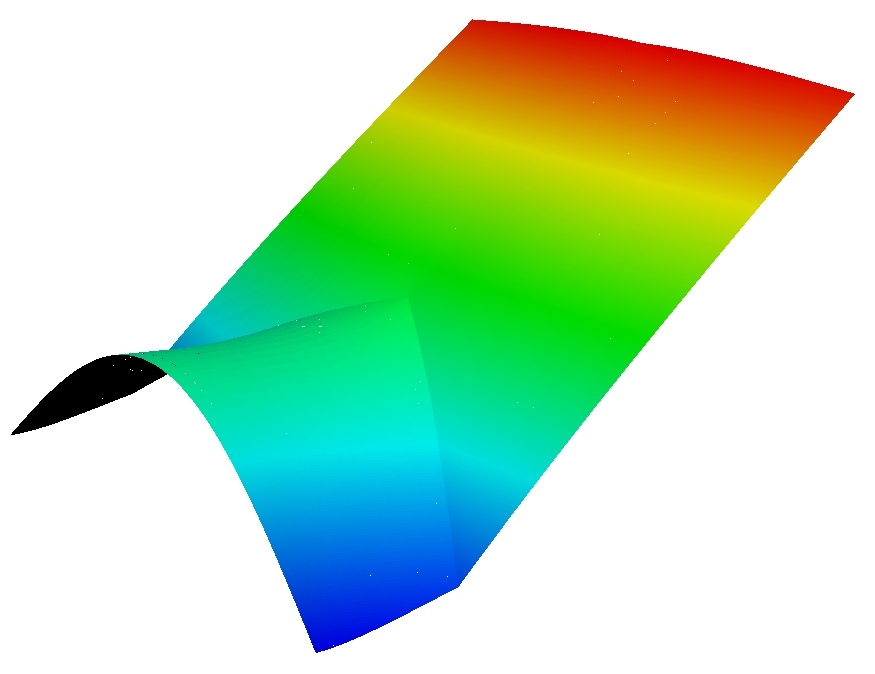}
       \includegraphics[width=0.08\textwidth]{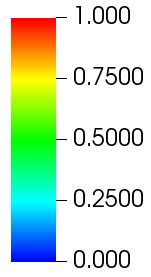} 
   \end{center}
\caption{
Adaptively refined mesh at level 10 (left) and 17 (middle), as well as the folding pattern (right) for a 
``V-shaped'' fold. 
}
\label{fig:example2_visuals}
\end{figure}
\begin{figure}
   \begin{center}
       \includegraphics[width=0.49\textwidth]{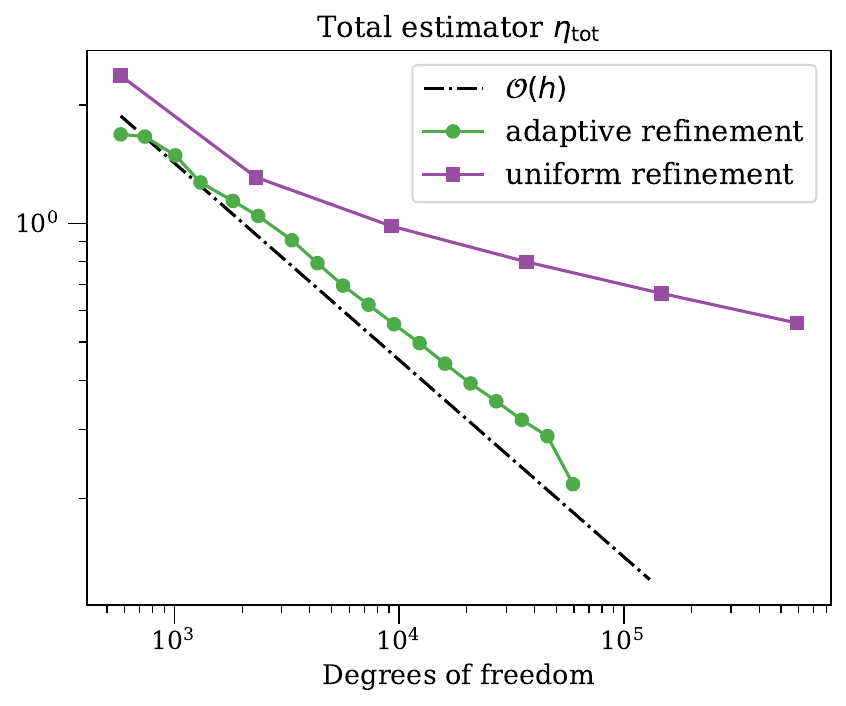}
       \includegraphics[width=0.49\textwidth]{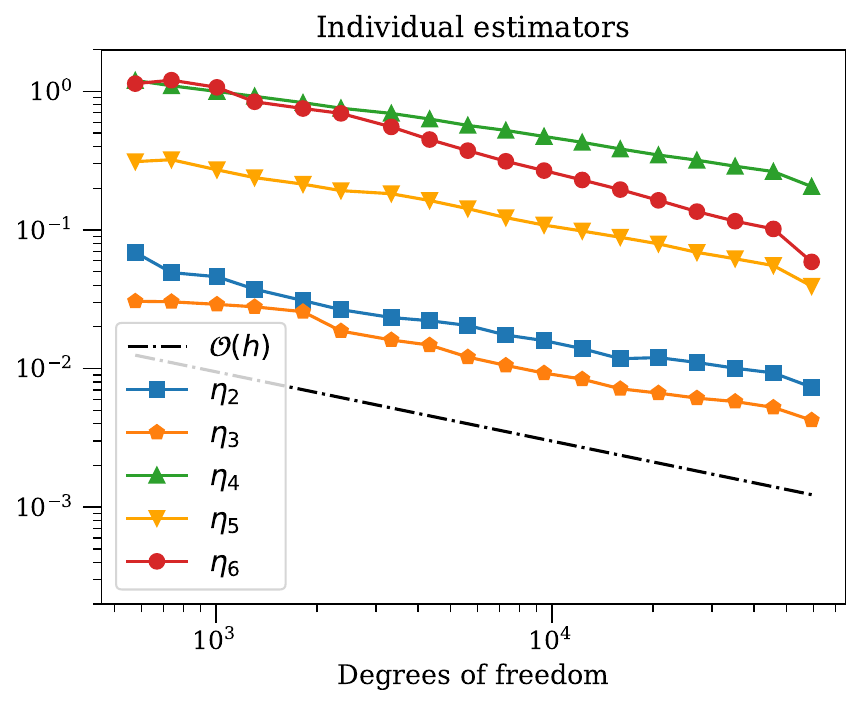}
   \end{center}
\caption{
Example 2: total estimator \eqref{eq:eta_tot_redefine} for the cases of adaptive and uniform refinement (left) and individual 
estimators, as defined by \eqref{eq:eta_defns} (right). 
}
\label{fig:example2_estimators}
\end{figure}

\subsection{Example 3: L-shaped domain and sinusoidal fold}
Lastly, consider the L-shaped domain $\Omega = (-1,1)^2 \setminus (0,1)^2$ and the sinusoidal fold 
parameterized as 
$\cC(x_1) = (x_1 , \frac16 \sin(\pi (x_1+1)) -0.5)$, 
where $x_1 \in (-1,1)$. 
We consider the inhomogeneous Dirichlet conditions 
$$
g(x,y) = \frac16(x^2 + y^2 + 2xy - x -y ) \quad \text{and} \quad \Phi(x,y) = \nabla g(x,y)
$$
for the entire boundary $\partial\Omega$
and  set $f(x_1,x_2) = 0$. 
The computational meshes at level 14 and 27 (left and middle), as well as 
the resulting folding pattern (right) are shown in 
Figure \ref{fig:example3_visuals}. 
In this case, the adaptive refinement leads to highly refined 
regions near both the fold $\cC$ and the corner singularity. 
As in the previous example, the analytic solution $u$ is not known. 
The nearly optimal convergence behaviour of the error estimator $\eta_{\rm tot}$ 
for adaptive refinement is shown in 
Figure \ref{fig:example3_estimators} (left); as expected, the rates are 
suboptimal for uniform refinement. 
The convergence plot for 
each individual estimator $\eta_i$, $2 \le i \le 6$, is also 
depicted in Figure \ref{fig:example3_estimators} (right).  
\begin{figure}
   \begin{center}
       \includegraphics[width=0.29\textwidth]{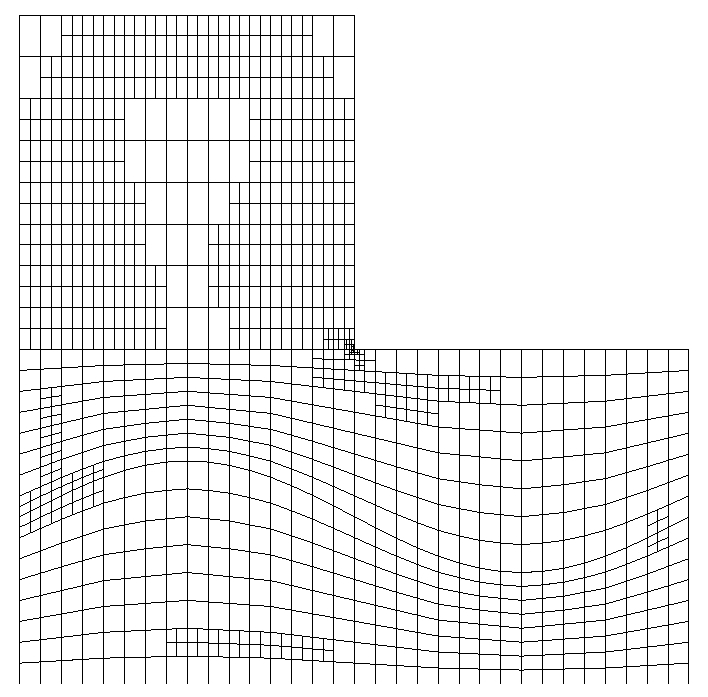}
\hskip-0.5em
       \includegraphics[width=0.29\textwidth]{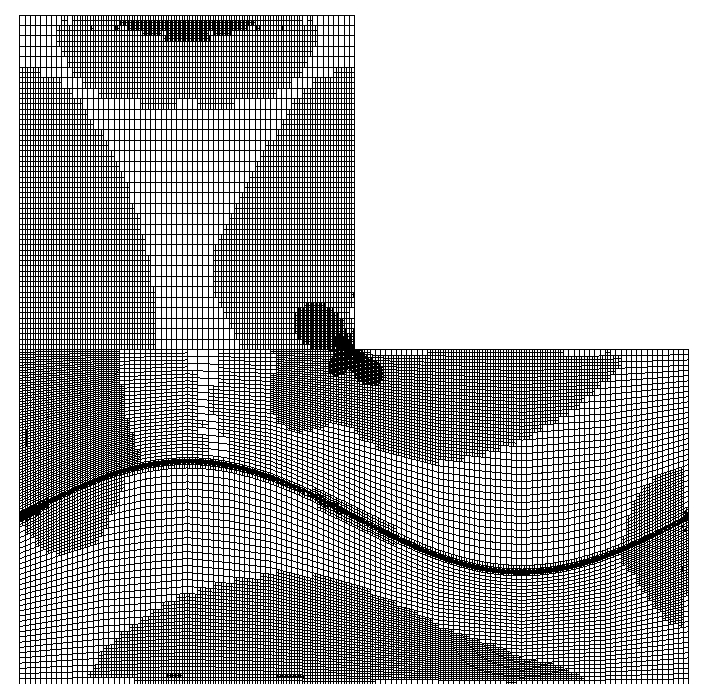}
       \includegraphics[width=0.32\textwidth]{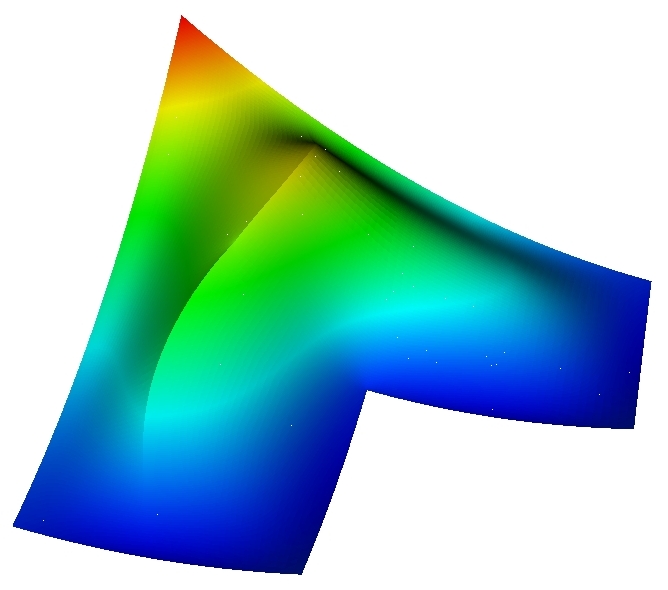}
       \includegraphics[width=0.08\textwidth]{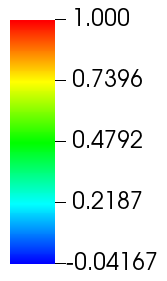} 
   \end{center}
\caption{
Adaptively refined mesh at level 14 (left) and 27 (middle), as well as the 
folding pattern (right) for an L-shaped domain 
and sinusoidal fold. 
}
\label{fig:example3_visuals}
\end{figure}
\begin{figure}
   \begin{center}
       \includegraphics[width=0.49\textwidth]{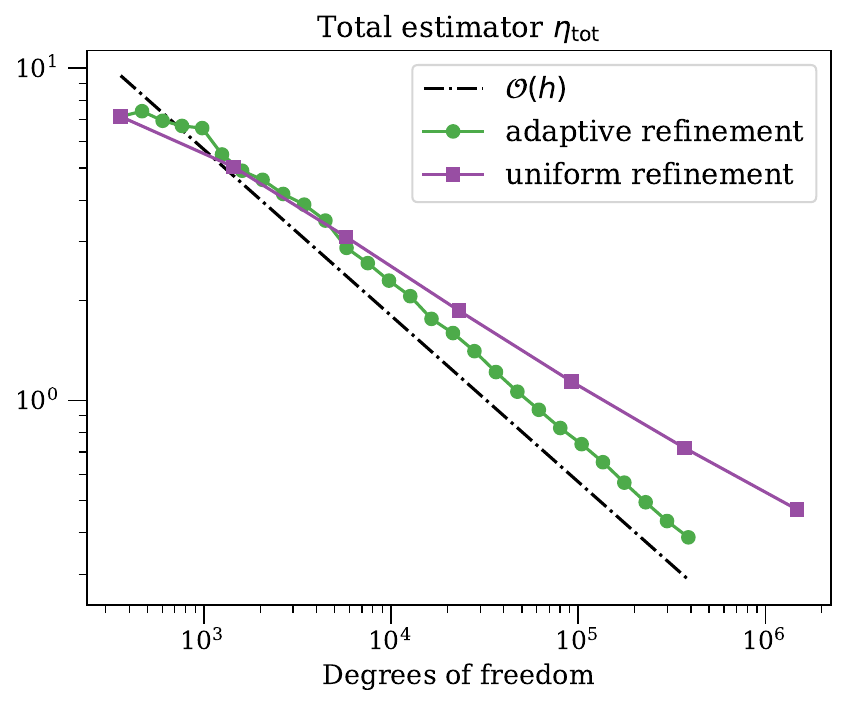}
       \includegraphics[width=0.49\textwidth]{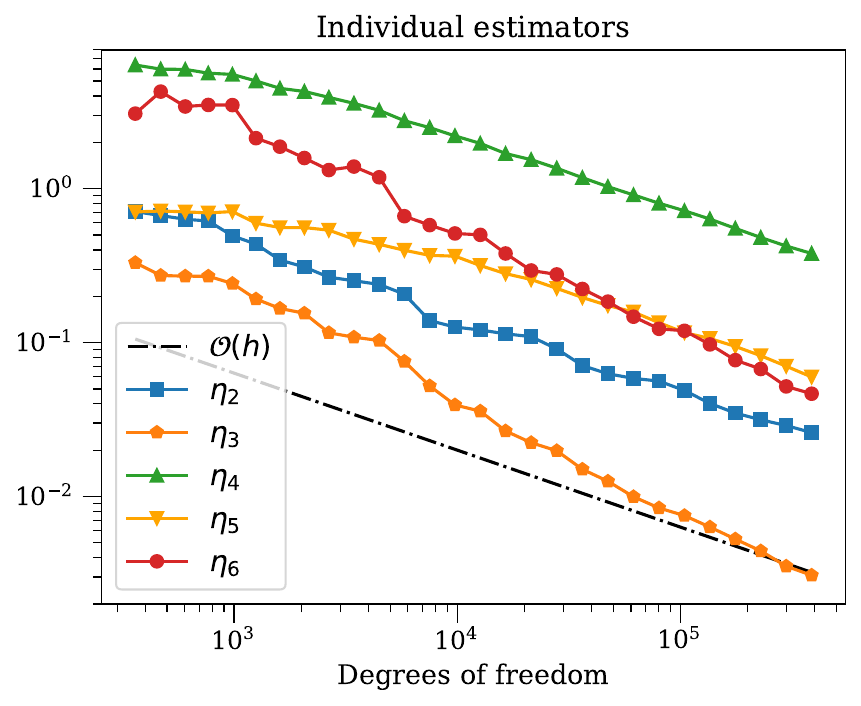}
   \end{center}
\caption{
Example 3: total estimator \eqref{eq:eta_tot_redefine} for the cases of adaptive and uniform refinement (left) and individual 
estimators, as defined by \eqref{eq:eta_defns} (right). 
}
\label{fig:example3_estimators}
\end{figure}

\section{Conclusions}
\label{sec:conclusions}
A fitted interior penalty discontinuous Galerkin method has been 
presented for a fourth order elliptic interface problem 
that arises from a linearized model of thin sheet folding.
A local efficiency bound for an estimator that measures
the extent to which the interface conditions along the fold 
has been proven. This required constructing a novel edge bubble function
that may be useful in the analysis of other interface problems. 
An improved \emph{a priori} error estimate under minimal 
solution regularity has also been shown via a \emph{medius} analysis. 
Numerical experiments illustrated the satisfactory performance 
of the \emph{a posteriori} bounds in practice.

\section*{Acknowledgments}
The authors thank Prof.\ S\"{o}ren Bartels for stimulating their interest 
in folding models during his visit to George Mason University in March 2024, as 
well as for suggesting the ``flapping'' numerical example in Section \ref{sec:numerical_examples}.
The authors also thank Keegan Kirk for insightful discussions. Finally, the authors gratefully acknowledge 
the community support from the \texttt{deal.II} User Group.


\bibliographystyle{amsplain}
\bibliography{references,references_RK}

\end{document}